\documentclass[12pt]{article}
\usepackage[paperheight=279.4mm,paperwidth=210mm,margin=1in]{geometry}

\usepackage{tocloft}
\usepackage{amsmath,amssymb,mathrsfs,tensor,bbold,upgreek,cases}
\usepackage{amsthm,amsfonts,mathtools,bm,slashed}
\usepackage{listings}
\usepackage{float,graphicx,comment,subfig}
\usepackage[hidelinks]{hyperref}
\usepackage{cite}
\usepackage[font=small,labelfont=bf]{caption}
\numberwithin{equation}{section}

\newcommand\mailto[1]{\href{mailto:#1}{\tt #1}}

\def\be{\begin{equation}}
\def\ee{\end{equation}}
\def\bea{\begin{align}}
\def\eea{\end{align}}
\newtheorem{theorem}{Theorem}[section]
\newtheorem{definition}[theorem]{Definition}

\theoremstyle{remark}

\DeclareMathOperator{\vol}{vol}

\newcommand{\M}{\mathcal{M}}
\newcommand{\N}{\mathcal{N}}

\newcommand{\E}{\mathcal{E}}
\newcommand{\R}{\mathcal{U}}

\newcommand{\Lie}{\mathcal{L}}

\newcommand{\order}{\mathcal{O}}
\newcommand{\bigo}{\mathcal{O}}

\newcommand{\X}{X}

\newcommand{\Me}{\bar{\mathcal{M}}}
\newcommand{\Mee}{\hat{\mathcal{M}}}
\newcommand{\Xt}{\tilde{X}}
\newcommand{\Hc}{\mathscr{H}}

\newcommand{\n}{n}
\newcommand{\m}{\ell}
\newcommand{\rr}{k}

\newcommand{\flow}{\varphi}
\newcommand{\flowN}{\phi}

\makeatletter
\def\blfootnote{\gdef\@thefnmark{}\@footnotetext}
\makeatother

\begin{document}
\thispagestyle{empty}

\renewcommand{\thefootnote}{\fnsymbol{footnote}}

\vspace*{1.5cm}
\begin{center}

{\bf{\Large
On dissipative symplectic integration with \\[.4em] applications
to gradient-based optimization
}}

\vspace{1cm}

{\bf Guilherme Fran\c ca,}$^{\!1,2,}\footnote{\mailto{guifranca@gmail.com}}$%
~{\bf Michael I. Jordan,}$^{\!1}$%
~{\bf and Ren\' e Vidal}$^{\,2}$%

\vspace{0.5cm}

${}^{1}${\it University of California, Berkeley, CA 94720, USA}\\[.2em]
${}^{2}${\it Johns Hopkins University, MD 21218, USA}

\vspace{2.0cm}

{\bf Abstract}
\vspace{-1em}
\end{center}

Recently, continuous-time dynamical systems have proved useful in providing conceptual and quantitative insights into gradient-based optimization, widely used in modern machine learning and statistics.  An important question that arises in this line of work is how to discretize the system in such a way that its stability and rates of convergence are preserved.
In this paper we propose a geometric framework in which such discretizations can be realized systematically, enabling the derivation of ``rate-matching'' algorithms without the need for a discrete convergence analysis. More specifically, we show that a generalization of symplectic integrators to nonconservative and in particular dissipative Hamiltonian systems is able to preserve rates of convergence  up to a controlled error.
Moreover,
such methods preserve a shadow Hamiltonian despite the absence of a
conservation law, extending key
results of symplectic integrators to
nonconservative cases.
Our arguments rely on a combination of backward error analysis with
fundamental results from symplectic geometry.
We stress that although the original motivation for this work
was the application to optimization, where dissipative
systems play a natural role,
they are fully general and not only provide a differential
geometric framework for dissipative
Hamiltonian systems but also substantially extend the theory of
structure-preserving integration.

\renewcommand*{\thefootnote}{\arabic{footnote}}
\setcounter{footnote}{0}
\setcounter{page}{0}
\newpage
\thispagestyle{empty}
\tableofcontents
\newpage

\section{Introduction}

A recent line of research at the interface of machine learning and optimization focuses on establishing connections between continuous-time dynamical systems and gradient-based optimization methods \cite{Candes:2016,Wibisono:2016,Franca:2018b,Franca:2018,Franca:2019,Krichene:2015,Wilson:2016,Scieur:2017,Betancourt:2018,Zhang:2018,Shi:2019,Muehlebach:2019,Diakonikolas:2019}.
From this perspective, an  optimization algorithm corresponds to a particular discretization of a differential equation.  Moreover,  important \emph{accelerated methods} such as Nesterov's method~\cite{Nesterov:1983} and Polyak's heavy ball method~\cite{Polyak:1964} are modeled as  second-order differential equations with a dissipative term \cite{Candes:2016,Wibisono:2016,Franca:2018,Franca:2018b,Franca:2019}.
There are advantages to working in continuous time. In particular, the stability and convergence analysis of a continuous system tends to be simpler and more transparent, making use of general tools such as Lyapunov stability theory and variational formulations. The traditional discrete analysis is usually only applicable on a case-by-case basis and often requires painstaking algebra. This is an unsatisfactory state of affairs given that accelerated optimization methods are the workhorses behind many of the empirical success stories in large-scale machine learning.

A particularly useful step has been the development of a variational perspective on acceleration methods~\cite{Wibisono:2016}. This framework, which involves the definition the so-called \emph{Bregman Hamiltonian}, places momentum-based methods such as Nesterov acceleration into a  larger class of dynamical systems and has accordingly helped to demystify the notion of ``acceleration'' in an optimization context.  Two difficulties arise, however, when one attempts to further exploit and characterize this class of systems. First, there are many different ways to discretize a continuous system.
A naive discretization may be unable to preserve rates of convergence---i.e., rates of decay to lower-energy level  sets---and may even lead to an  unstable algorithm.
Moreover, it is largely unknown if there exists an underlying principle from which one can construct such ``rate-matching'' discretizations.  Thus a fundamental question arises:
\begin{center}
\emph{Which classes of discretizations are capable of preserving  the  rates of convergence  \\ of the continuous-time dynamical systems of interest in optimization?}
\end{center}

A second difficulty is that the Hamiltonian formalism has traditionally been applied to conservative systems; in particular, systems characterized 
by oscillations.  Such
behavior is incommensurate with the desire to converge to an optimum---a system that converges towards a limit cycle may have favorable stability properties, but the presence of limit cycles may preclude convergence to a point.  Thus we have a second fundamental
question:
\begin{center}
\emph{Can we map discrete-time algorithms into dissipative continuous-time dynamical \\ systems  that provide analytical insight into the behavior of the original algorithm?}
\end{center}
Clearly these two questions are related.  Indeed, the ability to map between dynamical systems while preserving rates
may be seen as a form of ``invariance''---although different from the conservation laws arising from Noether's theorem.

In this paper we attempt to provide answers to the above  questions. Introducing a class of \emph{dissipative Hamiltonian systems}, we combine fundamental results from symplectic geometry \cite{Berndt,Aebischer} and backward error analysis \cite{Benettin:1994,Reich:1999,Hairer:1994} to establish a general quantitative guarantee for the convergence of discrete algorithms based on these continuous dynamics.
More specifically, we propose a class of discretizations, which we refer to as \emph{presymplectic integrators}, that are designed to preserve a fundamental geometric structure associated with nonconservative, and in particular dissipative, Hamiltonian systems.%
\footnote{This will be made precise later but briefly the idea is that the phase space of a nonconservative Hamiltonian system is a \emph{presymplectic manifold}, endowed with a closed \emph{degenerate} symplectic 2-form, which is exactly preserved by a presymplectic integrator.}

Presymplectic integrators consist of a generalization of
the well-known family of \emph{symplectic integrators}
\cite{Reich,Hairer,SanzSerna:1992,McLachlan:2006,Quispel:2018}
which have been developed in the setting of \emph{conservative} Hamiltonian systems. The most important property of symplectic integrators is that in addition to preserving the symplectic structure, they exactly conserve a perturbed or shadow  Hamiltonian \cite{Benettin:1994}, thus ensuring long-term stability. This crucial result relies on the fact that the Hamiltonian is a constant of motion.
On the other hand, there are relatively few results on structure-preserving methods for dissipative systems,
although this has been the subject of a nascent literature~\cite{Moore:2016,Moore:2017,Moore:2019,Shang:2020,Nicolis:2019,Nicolis:2019b}.
It is thus unknown if the key stability properties of symplectic integrators can be extended to dissipative cases precisely because a conserved quantity is no longer available.
We will show that presymplectic integrators allow such properties to be extended into a nonconservative setting. In particular, we show that they
preserve a---time-dependent---shadow  Hamiltonian and accordingly exhibit long-term stability.
Our argument relies on a \emph{symplectification} procedure where
the nonconservative system is embedded in the phase space of a higher-dimensional
conservative  system.

Although the principal goal of our work is to bring a dissipative physical systems perspective to gradient-based optimization,
we note that our results are fully general and not tied to
applications to optimization; rather, they yield a general differential geometric framework for the study of nonautonomous and dissipative Hamiltonian systems.  They also extend the existing theory of
structure-preserving integration to such cases. Our approach may therefore be of interest in other fields where the simulation of dissipative
systems is important, such as out-of-equilibrium statistical mechanics, thermodynamics of open systems,
complex systems, nonlinear dynamics, etc.

This paper is organized as follows.
In Appendix~\ref{diff_geo_sec} we introduce notation and recall the basic concepts from differential geometry that are needed throughout the paper.
In Section~\ref{summary} we provide a high level overview of the main outline of our analysis with a focus on the implications to optimization. Section~\ref{dissip_flow} introduces ideas from backward error analysis and symplectic integrators---for conservative systems---presenting independent geometric proofs of earlier results \cite{Benettin:1994,Hairer:1994,Reich:1999} so as to anticipate our generalizations to
dissipative systems.
In Section~\ref{ham_sys_sec} we introduce nonconservative Hamiltonian systems from the point of view of symplectic geometry
and construct their symplectification.
We then define presymplectic integrators  and argue that they extend the useful properties of symplectic integrators into nonconservative settings. In Section~\ref{implication_optimization}
we consider the implications of this framework for preserving convergence rates and stability of dissipative Hamiltonian systems, which in particular justify this approach for solving optimization problems.
In Section~\ref{bregman_sec} we construct  explicit presymplectic integrators for the Bregman Hamiltonian in full generality.
In Section~\ref{numerical} we provide numerical evidence that support our
theoretical results.
Section~\ref{conclusion} contains our our final remarks.

\section[Overview of the implications to accelerated  optimization]{Overview of the implications to \\ accelerated  optimization}
\label{summary}

Given an $n$-dimensional smooth manifold $\M$ and a function $f:\M \to \mathbb{R}$, consider the problem
\be \label{optimization}
f(q^\star) = \min_{q \in \M} f(q).
\ee
Let $H = H(t, q, p)$ be an explicitly \emph{time-dependent} Hamiltonian over the phase space $(q, p) \in T^\star \M$---the cotangent bundle of $\M$ (see
Appendix~%
\ref{diff_geo_sec})---
which determines dynamical evolution through Hamilton's equations:
\be \label{ham_motion}
\dfrac{d q^j}{dt} = \dfrac{\partial H}{\partial p_j}, \qquad \dfrac{dp_j }{dt} = - \dfrac{\partial H}{\partial q^j},
\ee
for $j=1,\dotsc,\n$.
We will design dissipative systems whose trajectories tend to a low energy level set that is consistent with a minimum  of $f$.%
\footnote{By way of contrast, a classical conservative system has a \emph{time-independent} Hamiltonian, $H = H(q,p)$, such that $\tfrac{dH}{dt} = 0$, implying that trajectories oscillate around a minimum instead of converging.}
Specifically, we consider systems arising from the following general family of Hamiltonians:
\be \label{gen_ham}
H \equiv e^{-\eta_1(t)} T(t,q,p) + e^{\eta_2(t)} f(q),
\ee
where $\eta_1(t)$ and $\eta_2(t)$ are positive and monotone increasing functions that are responsible for introducing dissipation.  The kinetic energy $T$ is assumed to be Lipschitz continuous. This Hamiltonian includes many dissipative systems that are relevant to optimization, including the Bregman Hamiltonian \cite{Wibisono:2016} and conformal Hamiltonian systems \cite{McLachlan:2001}---see
Appendix~%
\ref{generalized_conformal} for a generalization thereof.
Note also that \eqref{gen_ham} generalizes the Caldirola-Kanai Hamiltonian
\cite{Caldirola:1941,Kanai:1948} which
can be seen as the classical limit of the seminal Caldeira-Leggett model \cite{Caldeira:1981}, important
in quantum dissipation and decoherence.

It is possible to characterize the convergence rate that a dissipative system tends to a minimum  through a Lyapunov analysis \cite{Candes:2016,Wibisono:2016,Franca:2018b,Wilson:2016}. This leads to upper bounds of the general form
\be \label{gen_rate}
f(q(t)) - f(q^\star) = \bigo\left(\mathcal{R}(t)\right)
\ee
where $\mathcal{R}(t)$ is a decreasing function of time that depends on the
landscape of $f$.
One of our goals is to construct ``rate-matching'' discretizations, namely  general numerical integrators  able to reproduce \eqref{gen_rate}.
As previously mentioned, we propose a class of discretizations called  presymplectic integrators to this  end---this
is formalized in Definition~\ref{presymp_def} below.
%
%

Let  $x \equiv (q, p)$.
A  numerical integrator for the Hamiltonian system \eqref{ham_motion} is a  map $\flowN_h : \mathbb{R}^{2\n}\to\mathbb{R}^{2\n}$, with step size $h > 0$, such that iterations
\be\label{num_integrator}
x_\m = \flowN_h(x_{\m-1}), \qquad x_0 = x(0),
\ee
approximate the true
state $x(t_\m)\equiv (q(t_\m), p(t_\m))$ at  instants $t_\m = h \m$ ($\m = 1,2,\dotsc$).
Let $\flow_t$ denote the true flow of \eqref{ham_motion}. An  integrator $\flowN_h$ is said to be of
\emph{order} $r \ge 1$ if
\be \label{order}
\| \flowN_h(x) - \flow_h(x) \| = \bigo(h^{r+1}),
\ee
for any $x\in T^*\M$.  We also introduce a Lipschitz assumption for the integrator:
\be \label{lipschitz_flow}
\| \flowN_h(y) - \flowN_h(x) \| \le (1 + h L_\flowN) \| y - x\|
\ee
for some constant $L_\flowN > 0$ and for all $x,y \in T^*\M$.%
\footnote{This  condition is satisfied by a large class of methods, even including simple ones such as the explicit Euler method which does not preserve any dynamical invariant~\cite{Hairer,Reich}.}
We now state one of our main results, which we will further explicate and establish formally in the remainder of the paper.
\begin{theorem} \label{main_theorem}
Consider a dissipative Hamiltonian system \eqref{ham_motion} obtained from \eqref{gen_ham}. Let $\flowN_h$ be a presymplectic integrator of order $r$, assumed to obey the Lipschitz condition \eqref{lipschitz_flow}.  Then $\flowN_h$ preserves the continuous rates of convergence up to a small error, namely
\begin{equation} \label{main_rate}
  \underbrace{f(q_\m) - f(q^\star)}_{\textnormal{discrete rate}} = \underbrace{f(q(t_\m)) - f(q^\star)}_{\textnormal{continuous rate}} + \underbrace{\bigo\big(h^r e^{-\eta_2(t_\m)}\big)}_{\textnormal{small error}},
\end{equation}
provided
$e^{L_{\flowN}t_\m -\eta_1(t_\m)} < \infty$
for sufficiently large $t_\m$.
 This holds for exponentially large times $t_\m \equiv h \m = \bigo(h^r e^r e^{h_0/h})$, where the constant $h_0 > 0$ is independent of $h$.
\end{theorem}

This theorem shows that presymplectic integration can provide answers to the questions posed earlier
regarding the possibility of rate-matching discretizations of dissipative  systems.
The assumptions of the theorem are mild, and the restriction on $t_\m$ may be irrelevant in practice.
More importantly, the error in \eqref{main_rate} is small and improves with $r$, though it is dominated by $\eta_2$ which  suggests that, in this context, higher-order integrators are not likely to be needed.  Note that choosing a suitable $\eta_2$ is essential since it can make the error negligible; e.g., with $\eta_2 \sim t$ the error is exponentially small.

There is another important aspect of presymplectic integrators worth noting. Since they exactly preserve the phase space geometry, they reproduce the qualitative features of the phase portrait and in particular the stability of critical points.  This is not the case for typical discretizations, which in general introduce spurious damping or excitation.
In short, given any suitable dissipative system obtained from the Hamiltonian \eqref{gen_ham}, presymplectic integrators constitute
a general approach for the construction of optimization algorithms that are guaranteed to respect the stability and rates of convergence of the underlying dynamical system.  We carry out such an approach and consider explicitly the
case of the Bregman Hamiltonian in Section~\ref{bregman_sec}, and also provide general examples in Appendix~%
\ref{constructing_symp}.

\section[Conservative Hamiltonian systems and symplectic integrators]
{Conservative Hamiltonian systems and \\ symplectic integrators}
\label{dissip_flow}

In the remainder of the paper we provide a complete theoretical derivation justifying Theorem~\ref{main_theorem}. To build up to that derivation, we first recall several essential concepts from backward error analysis, symplectic geometry, and dynamical systems. These concepts are necessary for an understanding of how one can draw conclusions about structure-preserving methods without a discrete-time analysis. For further background on the basic differential geometry that we use to develop our ideas, we refer the reader to Appendix~%
\ref{diff_geo_sec}.

\subsection{Numerical integrators and modified equations}

Let $\M$ be an $\n$-dimensional smooth manifold, and let $(\R, x)$ be a chart such that every point $p \in \R\subset \M$ has local coordinates $x^1,\dotsc,x^\n$ in $\mathbb{R}^\n$. From now on we refer to a point $p$ by its coordinates $x$.
Given a vector field $X \in T\M$, where $T \M$ denotes the tangent bundle, 
one has a system of differential equations:
\be \label{autonomous}
\dfrac{d x^j }{dt} = \X^j(x), \qquad x^j(0) = x_0^j . 
\ee
This vector field can be represented by the differential operator
\footnote{We use Einstein summation convention throughout the paper.
}
\be \label{Xop}
X(x) = X^j(x) \partial_j,
\ee
where $\partial_1, \dotsc, \partial_\n$ is the induced coordinate basis in $T\M$. An integral curve of \eqref{autonomous} defines a flow $\flow_t : \M \to \M$. 
Notice that \eqref{autonomous} has unique solutions, at least locally, since $\X$ is locally Lipschitz due to the smoothness of $\M$. The flow can be represented by the exponential map
\be \label{expmap}
\flow_t = e^{t \X} = I + t \X + \tfrac{1}{2} t^2 \X^2 + \dotsm .
\ee
Through its pullback, denoted by $\flow_t^*$, the Lie derivative of a tensor $\alpha$ of rank $(0,q)$ along  $\X$ is defined by
\be \label{LieDer}
(\Lie_\X \alpha)(x)  \equiv \left.\dfrac{d}{dt}\right\vert_{t=0} \! \flow^*_t \alpha(x)
  = \lim_{t\to 0} \dfrac{ \flow_t^* \alpha(\flow_t(x)) - \alpha(x)}{t}.
\ee
The dynamical system is said to \emph{preserve} $\alpha$ if and only if $\Lie_X\alpha = 0$.

Let $\flowN_h: \mathbb{R}^\n \to \mathbb{R}^n$ be a numerical integrator of order $r\ge 1$ for the system \eqref{autonomous}---as defined in \eqref{num_integrator} and \eqref{order}.
Since $X$ is locally Lipschitz, we have $\| X(y) - X(x) \| \le L_X \| y - x \|$ for a constant $L_X > 0$ and for all $x,y \in \M$ in some region of interest. It follows from classical results \cite{Hairer,Reich} that there exists $L_{\flowN} > L_X$ such that we obtain control on a global error:
\be\label{glob_err}
\| \flowN_h^\m(x_0) - \flow_{t_\m}(x_0)\| \le C (e^{L_{\flowN} t_\m } - 1) h^{r},
\ee
for some constant $C > 0$ and initial state $x_0 \in \M$.  Note that we have denoted $\flowN_h^\m \equiv \flowN_h \circ \dotsm \circ \flowN_h$. Thus, for a fixed $t_\m$ the numerical method is accurate up to order $\bigo(h^{r})$. However, if  $t_\m$ is free the error grows exponentially. For this reason, although higher-order methods may provide accurate solutions in a short span of time, they can be inaccurate and unstable for large times.

Formally, every numerical integrator $\flowN_h$ can be seen as the \emph{exact flow} of a modified  or perturbed system \cite{Benettin:1994,Hairer:1994,Hairer:1997,Reich:1999,Hairer,Reich}:
\be  \label{modeqn}
  \dfrac{dx^j}{dt} = \Xt^{j}(x), \qquad \Xt \equiv X +  \Delta X_1 h +  \Delta X_2 h^2 + \dotsm,
\ee
where the $\Delta X$'s are expressed in terms of $X$ and its derivatives.
We refer to $\Xt$ as the perturbed or \emph{shadow vector field}.
In general, the series in \eqref{modeqn} is divergent, and it is necessary to consider a truncation \cite{Hairer:1997}. Following  \cite{Reich:1999}, suppose  we have found a truncation\footnote{Which exists for $\rr=r$ by assumption.}
\be \label{XtildeTrunc1}
\Xt_{\rr} = X +  \Delta X_1 h + \dotsm +  \Delta X_\rr  h^{\rr} \qquad (\rr\ge r)
\ee
such that $\flowN_h$ is an integrator of order $\rr$, namely
$\| \flowN_h(x) - \flow_{h, \Xt_{\rr} } (x) \| = \bigo(h^{\rr+1})$, where $\flow_{h,\Xt_{\rr}}$ denotes the exact flow of  \eqref{modeqn} with $\Xt$ replaced by $\Xt_{\rr}$.  Define
\be \label{DeltaXDef}
\Delta X_{\rr+1}(x) \equiv \lim_{h \to 0} \dfrac{\flowN_h(x) - \flow_{h, \Xt_\rr}(x)}{h^{\rr+1}}.
\ee
Then one can show that
\be \label{XtildeTrunc}
  \Xt_{\rr+1} \equiv \tilde{X}_\rr +  \Delta X_{\rr+1} h^{\rr+1}
\ee
yields a flow for which $\flowN_h$ is  an integrator of order $\rr+1$. 
Proceeding inductively, one can find higher-order vector fields $\tilde{X}_k$, with increasing $k$, such that $\flow_{h, \tilde{X}_k}$ becomes closer and closer to $\flowN_h$.
However,  since  \eqref{modeqn} eventually  diverges, there exists a truncation point $\rr^\star$ such that  $\| \flowN_h(x) - \flow_{h,\Xt_{\rr^\star}}(x)\|$ is as small as possible. 
Finding $\rr^\star$ and estimating the size of the $\Delta X$'s  terms is quite technical, but it has been carried out in seminal work~\cite{Benettin:1994}, and the approach has been further improved in subsequent literature \cite{Hairer:1997,Reich:1999}.

\begin{theorem}[see \cite{Reich:1999}] \label{error_bound}
Assume that $X$ is real analytic and bounded on a compact subset of its domain. Assume that the numerical method $\flowN_h$ is real analytic of order $r$. Then, there exists a family of shadow vector fields $\Xt$ such that
\be \label{boundXtilde}
\big\| X(x) -\Xt(x) \big\| = \bigo( h^r), \qquad
\big\|  \flowN_h(x) - \flow_{h, \Xt}(x) \big\| \le C h e^{-r} e^{- h_0 / h},
\ee
where the constants $C,h_0 > 0$ do not  dependent on the step size $h$.
\end{theorem}

This result holds in full generality, for any \emph{autonomous} dynamical system and any numerical integrator.
Next we will discuss how it becomes particularly useful in the case of conservative Hamiltonian systems.

\subsection{Conservative Hamiltonian systems} \label{ham_cons}

Before talking about symplectic integrators, we need to introduce conservative Hamiltonian systems.
Hamiltonian systems are ubiquitous because they are naturally attached to the geometry of the cotangent bundle $T^*\M$ of any differentiable manifold $\M$~\cite{Berndt,Aebischer}. We provide a concise introduction to key results that will be necessary later.

\begin{definition} \label{symp_mani} Let $\M$ be an even-dimensional smooth manifold supplied with a closed nondegenerate 2-form $\omega$. More precisely:
\begin{enumerate}
\item $d\omega = 0$.
\item On any tangent space $T_x\M$, if $\omega(X, Y) = 0$ for all $Y \ne 0 \in T_x\M$, then $X = 0$.
\end{enumerate}
Then $\omega$ is called a \emph{symplectic structure} and $(\M, \omega)$ a \emph{symplectic manifold}.
\end{definition}

\begin{theorem} \label{cotangent_symplectic}
Let $q^1,\dotsc,q^\n$ be local coordinates of $\M$ and
 $q^1,\dotsc,q^\n,p_1,\dotsc,p_\n$ the corresponding induced coordinates of the cotangent bundle $T^*\M$.
Then $T^*\M$ admits a closed  nondegenerate symplectic structure,
\be \label{symp}
\omega = dp_j \wedge dq^j,
\ee
and is therefore a symplectic manifold.
\end{theorem}
\begin{proof}
There always exists a globally defined Liouville-Poincar\' e 1-form,
$\lambda \equiv p_j dq^j \in T^*\M$ \cite{Berndt}.
Applying the exterior derivative  induces the Poincar\' e 2-form
$\omega \equiv d \lambda = dp_j \wedge dq^j$.
Since $d^2 = 0 $ (see \ref{poinc_lemma}) we
trivially have $d\omega = 0$. It is also easy to see that $\omega$ is nondegenerate; e.g., in a matrix representation $\omega = \left( \begin{smallmatrix} 0 & -I \\ +I & 0 \end{smallmatrix} \right)$, thus $\det ( \omega )= 1 \ne 0$.
\end{proof}

\begin{theorem}\label{Ham_omega}
A dynamical system with  phase space $T^*\M$ preserves the symplectic structure \eqref{symp} if and only if it is (locally) a conservative Hamiltonian system, namely has the form in~\eqref{ham_motion} with a time-independent Hamiltonian $H = H(q,p)$.
\end{theorem}
\begin{proof}
Let $\gamma : \mathbb{R} \to \M$ be a curve parametrized by $t$, so that $q^j(t) \equiv (q^j \circ \gamma)(t)$ and $p_j(t) \equiv (p_j \circ \gamma)(t)$ are time-dependent coordinates over $T^*\M$.
Consider the tangent vector $X \equiv \tfrac{d}{dt}$ to this curve, which in a coordinate basis is given by
\be
X = \dfrac{dq^j}{dt} \dfrac{\partial}{\partial q^j} +
\dfrac{dp_j}{dt} \dfrac{\partial }{\partial p_j} .
\ee
Note that $X$ lives in the tangent bundle of the phase space $T^*\M$.
The vector field $X$ preserves \eqref{symp} if and only if
$\Lie_{X} \omega = 0$.
Recalling Cartan's magic formula,
\be\label{cartan}
\Lie_X = d\circ i_X + i_X \circ d,
\ee
we conclude that $(d\circ i_X)(\omega) = 0$. The Poincar\' e lemma thus implies the existence of a (local) Hamiltonian function $H : T^*\M \to \mathbb{R}$, such that
\be \label{ham1}
  i_X(\omega) = - dH.
\ee
This is actually Hamilton's equations \eqref{ham_motion} in disguise; indeed, in component form we have
$i_X(\omega) =  \dot{p}_j dq^j - \dot{q}^j dp_j$ which by comparison with \eqref{ham1} yields
\eqref{ham_motion}.
We have just shown that a vector field that preserves the symplectic form \eqref{symp} genererates Hamiltonian dynamics.
It is now easy to show the converse, namely that the flow of a Hamiltonian system preserves the 2-form \eqref{symp}.  Given a Hamiltonian $H$, we have the equations of motion \eqref{ham_motion} with an associated vector field $X_H$. These equations can equivalently be written in the form \eqref{ham1}, as already shown. Using the identities $d^2H = 0$ and $d\omega = 0$ in \eqref{cartan} implies  $\Lie_{X_H} \omega = 0$.
\end{proof}


Finally, another fundamental property of Hamiltonian systems is energy conservation, which follows immediately  from the equations of motion \eqref{ham_motion}:
\be \label{Hconst}
\dfrac{dH}{dt} =
0 .
\ee
Thus, any conservative Hamiltonian system have two fundamental properties:
its flow preserve the symplectic structure, and
the Hamiltonian is a constant of motion.

\subsection{Symplectic integrators} \label{symp_sec}

Consider the following class of numerical methods.

\begin{definition} \label{symp_int_def}
A numerical integrator $\flowN_h$ for a conservative Hamiltonian system is a \emph{symplectic integrator}
if it preserves \eqref{symp}, i.e., $\flowN_h^* \circ \omega \circ \flowN_h = \omega$ where
$\flowN_h^*$ is the pullback.
\end{definition}

It is now easy to see that the flow of the perturbed system \eqref{modeqn} associated to a symplectic integrator
also preserves $\omega$ exactly if it is Hamiltonian. As a consequence, Theorem~\ref{Ham_omega} implies that the perturbed  system must be Hamiltonian, i.e.,
there exists a shadow Hamiltonian $\tilde{H}$ which is a perturbed version of $H$.
Indeed, since $\flowN_h$ preserves $\omega$ by assumption, and the flow of the perturbed system \eqref{modeqn} is an exact description of $\flowN_h$, the
shadow vector field $\Xt$ obeys $\Lie_{\Xt} \omega = 0$. Replacing the expansion \eqref{modeqn} and using the linearity of the Lie derivative in the vector field implies $\Lie_{\Delta X_j} \omega = 0$ for each $j = 1,2,\dotsc$.
Hence, Theorem~\ref{Ham_omega} implies that not only $\Xt$ but all $\Delta X_j$'s are Hamiltonian vector fields.  That is,
by the same argument leading to \eqref{ham1}, there exists a function $\tilde{H}$ and functions $H_j$'s such that
\be
i_{\Xt}(\omega) = - d\tilde{H}, \qquad  i_{\Delta X_j}(\omega) = - d H_j. 
\ee
Using the series \eqref{modeqn}, combining these two equations, and using \eqref{ham1}, we have that
\be
\begin{split}
-d \tilde{H} &= i_{X}(\omega) + \sum_j h^j i_{\Delta X_j}(\omega) = - dH - \sum_j h^j d H_j,
\end{split}
\ee
or equivalently
\be \label{shadow_ham}
\tilde{H} = H + h H_1 + h^2 H_2 + \dotsm .
\ee
Moreover, the shadow Hamiltonian $\tilde{H}$ is exactly conserved by the perturbed system, or equivalently
by the symplectic integrator: $\tilde{H} \circ \flowN_h  = \tilde{H}$. If we now consider a truncation~\eqref{XtildeTrunc}, or equivalently if we truncate \eqref{shadow_ham}, then the flow $\flow_{h, \Xt_k}$ exactly conserves
the truncated shadow Hamiltonian $\tilde{H}_k$:
\be \label{ham_k_flow}
\tilde{H}_k \circ \flow_{h, \Xt_k} = \tilde{H}_k.
\ee
We are now ready to state the most important property of symplectic integrators.
Adapting ideas from \cite{Benettin:1994} we can prove the following.
\begin{theorem}\label{symp_ham}
Let $\flowN_h$ be a symplectic integrator  of order $r$. Assume that the Hamiltonian $H$
is Lipschitz. Then $\flowN_h$ conserves $H$  up to
\be \label{HamError}
H \circ \flowN_h^\m = H\circ \flow_{\m h}  + \bigo(h^r),
\ee
for exponentially large times $t_\m = \bigo(h^r e^r e^{h_0/h})$.
Recall that $\flow_{\m h}$ is the true flow  
and $H = H\circ \flow_{\m h}$ is conserved, thus
$H(q_\ell, p_\ell) = H(q_0, p_0) + \bigo(h^r)$.
\end{theorem}
\begin{proof}
Let $\Xt \equiv \Xt_k$ be a truncated shadow vector field associated to $\flowN_h$---in what follows we omit  $k$ for simplicity.
We already know that \eqref{ham_k_flow} holds true, hence
\be \label{sum1}
\tilde{H}\circ \flowN_h^\m - \tilde{H}
= \sum_{i=0}^{\m-1} \left( \tilde{H}\circ\flowN_h^{i+1} - \tilde{H}\circ \flowN_h^{i} \right)
= \sum_{i=0}^{\m-1} \left( \tilde{H}\circ\flowN_h - \tilde{H}\circ \flow_{h,\tilde{X}}\right)\circ \flowN_h^{i},
\ee
where the first equality is an identity.
Let $\tilde{L}$ be the Lipschitz constant of $\tilde{H}$. Using the second relation of \eqref{boundXtilde} we find  that
\be \label{QExpSmall}
\big| \big(\tilde{H}\circ \flowN_h^\m - \tilde{H} \big)(x_0) \big|  \le
\tilde{L} \sum_{i=0}^{\m-1} \| \flowN_{h}(x_{i}) - \flow_{h,\Xt}(x_{i}) \|
\le \tilde{L} C  \m h \, e^{-r} e^{-h_0/h},
\ee
so $\flowN_h$ preserves  $\tilde{H}$ up to an exponentially small error in $h^{-1}$. In order to approximate $H$, write \eqref{shadow_ham} in the form $\tilde{H} = H + \bigo(h^r)$---since $\flowN_h$ has order $r$---to obtain \eqref{HamError}. Note that to ensure  the contribution from \eqref{QExpSmall} remains smaller than $\bigo(h^r)$ we have to choose
$t_\m = h \m$ such that $t_\m \, e^{-r} e^{-h_0/h} \sim h^r$.
\end{proof}

Theorem~\ref{symp_ham} explains the benefits of symplectic integrators.  Besides preserving the symplectic structure of the system  exactly, such methods generate solutions that remain within $\bigo(h^r)$ of the true energy surface of the system and thus exhibit long-term stability. It is well-known that in practice symplectic integrators tend to outperform alternative approaches when simulating conservative Hamiltonian systems \cite{SanzSerna:1992,McLachlan:2006,Quispel:2018}.

The situation is quite different for nonconservative or dissipative
Hamiltonian systems.  Even if one applies a symplectic integrator to the system written in  the extended phase space, 
\emph{the Hamiltonian is no longer a conserved quantity}. This conservation law is the most basic assumption underneath Theorem~\ref{symp_ham}; if \eqref{Hconst} is no longer true, then \eqref{ham_k_flow} is no longer true, and the argument leading to \eqref{sum1} and \eqref{QExpSmall} breaks down.
Therefore, it is not guaranteed that applying a symplectic integrator to a nonconservative Hamiltonian system will closely reproduce the Hamiltonian, which is varying over time, nor is it guaranteed that the method will exhibit long-term stability. One of the main contributions of the current paper is to show that Theorem~\ref{symp_ham} can be extended to general nonconservative Hamiltonian systems 
despite the nonexistence of such a conservation law.

\section{Nonconservative Hamiltonian systems} \label{ham_sys_sec}

Consider a time-dependent Hamiltonian $H = H(t,q,p)$.\footnote{As probably evident by now, we often refer to dissipative Hamiltonian systems although the reader should keep in mind that everything we say actually holds for any explicitly
time-dependent or nonconservative Hamiltonian system.}
The evolution of the system is still governed by  Hamilton's equations \eqref{ham_auto}, however the energy conservation law \eqref{Hconst} no longer holds and is replaced by
\be \label{Hnonconst}
\dfrac{dH}{dt} =
\dfrac{\partial H}{\partial t} .
\ee
We begin by introducing a $(2\n+1)$-dimensional extended phase space $T^*\Me$  in which time becomes a new coordinate, so that  $q^0=t, q^1, \dotsc, q^\n, p_1, \dotsc p_\n$ are local coordinates of $T^* \Me$. The evolution of the system is now generated by the vector field
\be \label{HamVec}
X \equiv X^\mu \partial_\mu = \dfrac{\partial}{\partial q^0} +
\dfrac{dq^j}{ds} \dfrac{\partial}{\partial q^j} +
\dfrac{dp_j}{ds} \dfrac{\partial}{\partial p_j},
\ee
where $\mu = 0,1,\dotsc,\n$, $X^0 = 1$,  and $s$ denotes  the ``new time parameter.'' Thus, the \emph{nonautonomous} Hamiltonian system \eqref{ham_motion}   is equivalent to the autonomous system
\be \label{ham_auto}
\dfrac{dq^0}{ds} = 1, \qquad \dfrac{dq^j}{ds} = \dfrac{\partial H}{\partial p_j}, \qquad
\dfrac{dp_j}{ds} = - \dfrac{\partial H}{\partial q^j},
\ee
over $T^*\Me$, where $H = H(q^0,\dotsc,q^n,p_1,\dotsc,p_n)$ is independent of $s$.
Importantly,  $T^*\Me$ is odd-dimensional so this phase space is no longer a symplectic manifold.

We thus introduce another dimension by adding a conjugate momentum $p_0$ that pairs with $q^0=t$. To make the distinction clear, we denote the $(\n+1)$-dimensional configuration manifold by $\Mee \equiv \mathbb{R}\times \M$, which has local coordinates $q^0,\dotsc,q^\n$.  The associated cotangent bundle $T^*\Mee$ is now of even dimensionality ${2\n+2}$, with  coordinates $q^0,\dotsc,q^\n,p_0,\dotsc,p_n$. Therefore,  $T^*\Mee$ has a natural  Liouville-Poincar\'e \mbox{1-form}
$  \Lambda = p_\mu dq^\mu$
which, as in the proof of Theorem~\ref{cotangent_symplectic}, induces a closed nondegenerate Poincar\' e 2-form:
\be \label{poince2}
\Omega \equiv d\Lambda = dp_\mu \wedge dq^\mu .
\ee
This makes $T^*\Mee$ a proper \emph{symplectic manifold}.
Requiring that the vector field
\be \label{newVec}
Y \equiv Y^\mu \partial_\mu = \dfrac{dq^\mu}{ds}\dfrac{\partial}{\partial q^\mu} + \dfrac{dp_\mu}{ds}\dfrac{\partial}{\partial p_\mu}
\ee
preserves the symplectic structure \eqref{poince2} implies---by the argument leading to \eqref{ham1}---that
\be \label{ham_not_yet}
  i_Y \Omega = -d\Hc
\ee
for some Hamiltonian  $\Hc: T^*\Mee \to \mathbb{R}$. In components, we have the following Hamilton's equations:
\be \label{ham_big}
  \dfrac{dq^0}{ds} = \dfrac{\partial \Hc}{\partial p_0}, \qquad
  \dfrac{dp_0}{ds} = -\dfrac{\partial \Hc}{\partial q^0}, \qquad
\dfrac{dq^j}{ds} = \dfrac{\partial \Hc}{\partial p_j}, \qquad
\dfrac{dp_j}{ds} = -\dfrac{\partial \Hc}{\partial q^j},
\ee
where $j=1,\dotsc,n$ are spatial indexes. Moreover, since now $\Hc$ does not depend explicitly on time $s$, the above system is actually conservative:
\be \label{cons_ext}
\dfrac{d \Hc}{ds} = 0.
\ee

At this stage, the higher-dimensional system \eqref{ham_big} is not equivalent to \eqref{ham_auto}.
In order to create this equivalence we need to impose constraints, which means \emph{fixing a gauge}.
Note that the symplectic form \eqref{poince2} is $\Omega = dp_0\wedge dq^0 + \omega$, where $\omega$ is the symplectic form \eqref{symp} for the original phase space $T^*\M$.
Requiring the vector fields \eqref{newVec} and \eqref{HamVec} to be the same, and using \eqref{ham_not_yet} together with \eqref{ham1}, yields
\be
0 = i_{Y} \Omega + d \Hc = i_{X} dp_0 \wedge dq^0 +  i_{X} \omega +  d\Hc = d( -p_0 - H + \Hc).
\ee
Thus, up to an irrelevant constant, we must have
\be \label{surf}
\Hc(q^0,\dotsc,q^\n,p_0,\dotsc,p_\n) = p_0 + H(q^0,\dotsc,q^\n,p_1,\dotsc,p_\n).
\ee
Note that we carefully made the variable dependencies of each term explicit.
This equation defines an \emph{embedded submanifold} in the  symplectic manifold $T^* \Mee$.
Hence,
the dynamics of the nonconservative Hamiltonian system \eqref{HamVec} lies on a hypersurface of constant energy $\Hc$. 
More precisely, from \eqref{surf} the first equation in \eqref{ham_big} gives
\be \label{q0}
\dfrac{dq^0}{ds} = 1,
\ee
which together with the two last equations of \eqref{ham_big} becomes precisely the original  system \eqref{ham_auto} over the extended phase space $T^*\Me$. The second equation of \eqref{ham_big} reproduces the dissipation
given by \eqref{Hnonconst}:
\be \label{fourth}
\dfrac{dp_0}{ds} = - \dfrac{\partial H}{\partial q^0} = -\dfrac{dH}{ds},
\ee
where the second equality follows from \eqref{cons_ext} together with \eqref{surf}.
Hence, up to another irrelevant constant, we have
\be \label{p0}
  p_0(s) = - H(s),
\ee
with $H(s) = H(q(s), p(s))$. Note that in \eqref{p0} the Hamiltonian is solely a function of time---the actual trajectories have been replaced---therefore $p_0$ is completely fixed.

We remark that \eqref{q0} and \eqref{p0} are specific choices of coordinates on $T^*\Mee$ which remove the spurious degrees of freedom that are not present in the original system.   This procedure of embedding the phase space of the nonconservative Hamiltonian system \eqref{ham_auto}, which does not admit a symplectic structure, into a higher-dimensional symplectic manifold is called \emph{symplectification}~\cite{Berndt,Aebischer}. We provide
an illustration in Fig.~\ref{submanifold}.

Let us comment on another point behind the gauge choice \eqref{q0} and \eqref{p0}.
On the extended phase space $T^* \Me$ it is still possible to distinguish between position $q^0$ and momentum $p_0$. The flow $\flow_s = e^{s X}$ defines orbits $s \mapsto q^\mu(s)$ on $\Mee$. Such  orbits are considered equivalent by time reparametrization, $s\mapsto s'(s)$. 
However, to match the original orbits of the  time-dependent Hamiltonian system over $\M$, one must fix $s = q^0 = t$. This means \emph{fixing a reference frame} on $T^*\Me$.
For this reason, time-dependent Hamiltonian systems are not covariant and one \emph{is not free to reparametrize the original time $t$}.
In other words, since the Hamiltonian system is explicitly time-dependent, time transformations are not canonical.

\begin{figure}
\centering
\includegraphics[scale=.27]{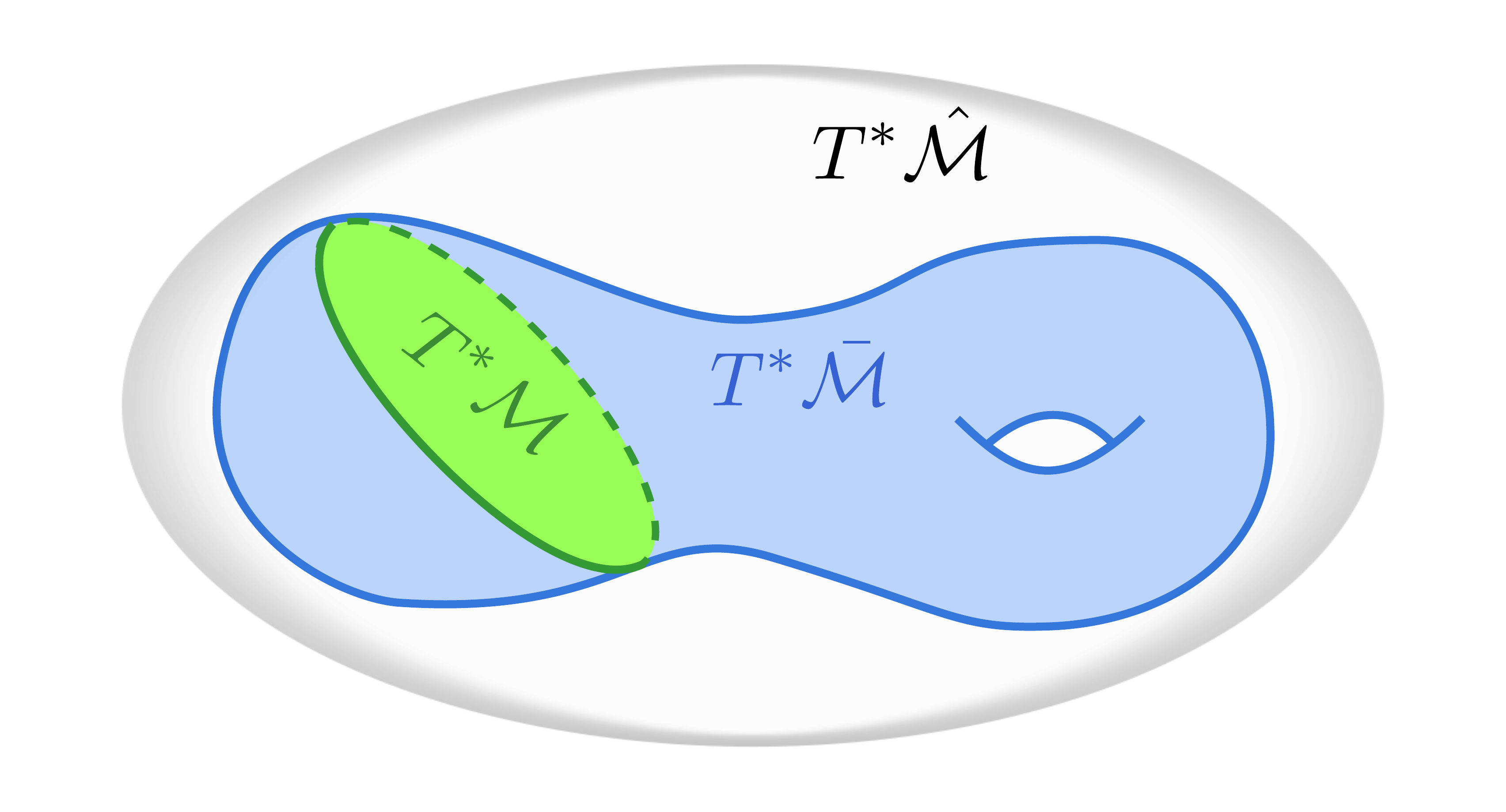}
\caption{\label{submanifold}
The symplectification consists of the embeddings $T^*\M \hookrightarrow T^*\bar{\mathcal{M}} \hookrightarrow T^*\hat{\mathcal{M}}$,  each of dimension $2\n$, $2\n+1$, and $2\n+2$, respectively.
The original dissipative (or nonconservative) Hamiltonian system has phase space $T^*\M$. When the system is written in autonomous form it has phase space $T^*\bar{\mathcal{M}}$, which is a presymplectic manifold that can naturally be embedded as a submanifold of constant energy $\Hc = 0$ in the symplectic manifold $T^*\hat{\mathcal{M}}$,
which can be associated to the phase space of a conservative Hamiltonian system.
}
\end{figure}

\subsection{Presymplectic manifolds}  \label{presymppp}

We are now in a position to delineate the specific geometry underlying nonconservative Hamiltonian systems.
First, we introduce a useful generalization of symplectic manifolds~\cite{Berndt}.

\begin{definition}\label{presymp_def} Let $\M$ be a differentiable manifold of dimension $(2\n + \bar{\n})$, $\bar{\n} \ge 0$, supplied with a 2-form $\omega$ of rank $2\n$ everywhere. Then $\omega$ is called a \emph{presymplectic form} and
  $(\M, \omega)$ a \emph{presymplectic manifold}.
\end{definition}

When $\bar{\n}=0$ this definition reduces to that of a symplectic manifold, and when $\bar{\n}=1$ it reduces to the definition of a \emph{weak contact manifold}.  A weak contact manifold supplied with a 1-form $\vartheta$ obeying $\vartheta\wedge (d\vartheta)^\n \ne 0$ is a \emph{contact manifold} \cite{Berndt,Aebischer}. There is an alternative theory of contact manifolds that are intimately related to symplectic manifolds. In particular, they correspond to a particular case of presymplectic manifolds and can naturally be embedded as hypersurfaces in their symplectification, which is a symplectic manifold in $2\n+2$ dimensions that is associated to $\M$.  
Furthermore, in Appendix~%
\ref{generalized_conformal} we show that a nonautonomous generalization of the so-called conformal Hamiltonian systems \cite{McLachlan:2001} also corresponds to a particular case of the time-dependent Hamiltonian formalism discussed above.

In order to project the symplectic form \eqref{poince2} into the hypersurface defined by \eqref{surf} we can simply
substitute the gauge choice $q^0=t$ and $p_0 = -H$ to  obtain
\be \label{newOmega}
  \Omega = - dH \wedge dt + \omega,
\ee
where we recall that $\omega$ is the original symplectic form of $T^*\M$.
Moreover, the equations of motion \eqref{ham_not_yet} reduce to
\be \label{iom_time}
i_{X} \Omega = 0, \qquad i_{X} dt = 1.
\ee
The vector field $X$, whose flow $\flow_s = e^{s X}$ generates  dynamical evolution on $\Mee$, is thus a \emph{zero mode} of $\Omega$. Note that, from Cartan's formula \eqref{cartan}, if we require that a vector field $X$ obeys \eqref{iom_time}, then it immediately
implies  $\Lie_X \Omega = 0$ so that $\Omega$ is preserved.
The symplectic form $\Omega$ is closed, however when projected into $T^*\Me$ it becomes \emph{degenerate}. One can see this through a  matrix representation
\cite{Sternberg}:
\be \label{OmDeg}
\Omega = \left( \begin{array}{c|cc} 0 & 0 & 0 \\ \hline 0 & 0 & -I \\  0 & +I & 0  \end{array}  \right)
\ee
which has a vanishing determinant. Therefore, the phase space of a nonconservative or dissipative Hamiltonian system written in the
autonomous form \eqref{ham_auto} is a presymplectic manifold (see again Fig.~\ref{submanifold}).

\subsection{Presymplectic integrators} \label{presymplectic_int}

Since a nonconservative system \eqref{ham_auto} admits a symplectification,  we can construct structure-preserving discretizations by imposing constraints on a \emph{symplectic integrator}. More precisely,  we can apply any standard symplectic integrator to the higher-dimensional conservative system \eqref{ham_big}. Then by the gauge choice in \eqref{q0} and \eqref{p0} we obtain an integrator for system \eqref{ham_auto}. Such a method will preserve a presymplectic structure. Accordingly, we refer to such integrators as \emph{presymplectic integrators}.

\begin{definition} \label{presymp_int} A numerical  map $\flowN_h$ is a \emph{presymplectic integrator} for a nonconservative  Hamiltonian system if it is obtained from a symplectic integrator for its symplectification under the gauge fixing \eqref{p0} and \eqref{q0}.
\end{definition}

Since the Hamiltonian $\Hc$ of the higher-dimensional system \eqref{ham_big} is conserved, we can apply standard results for symplectic integrators to derive conclusions about presymplectic integrators and thereby derive properties of systems without an underlying conservation law.

\begin{theorem} \label{preserv_ham_decay}
A presymplectic integrator $\flowN_h$ of order $r$ preserves the explicitly time-dependent Hamiltonian (assumed to be Lipschitz) up to
\be \label{ham_ext}
H \circ \flowN^\m_h   = H \circ\flow_{\m h}  + \bigo(h^r),
\ee
for exponentially large simulation times $s_\m = h \m = \bigo(h^r e^r e^{h_0/h})$.
\end{theorem}
\begin{proof}
Let $\Phi_h : \mathbb{R}^{2\n+2} \to \mathbb{R}^{2\n+2}$ be a symplectic integrator for the high-dimensional Hamiltonian system \eqref{ham_big}. Since $\Hc$ is conserved, Theorem~\ref{symp_ham} implies
\be \label{symp_ext}
\Hc \circ \Phi^\m_h  = \Hc\circ\Psi_{\m h} + \bigo(h^r),
\ee
with $s_\m = \bigo(h^r e^r e^{h_0/h} )$ and where $\Psi_{s}$ denotes the true flow.   With the gauge choice \eqref{q0} and \eqref{p0} the system is projected into the hypersurface \eqref{surf}, thus $\Psi_s$ reduces to $\flow_{s}$, which is the true flow of \eqref{ham_auto}, and  $\Phi_h$ reduces to
$\flowN_h:\mathbb{R}^{2\n+1} \to \mathbb{R}^{2\n+1}$, which approximates $\flow_{s}$. Thus, using these maps together with the substitution of \eqref{surf} into  \eqref{symp_ext}, we conclude:
\be
H\circ \flowN_h^\m + p_0(\m h) = H\circ\flow_{\m h} + p_0( \m h) + \bigo(h^r).
\ee
It is important to recall that $p_0$ is a fixed function of time and this is why it can be canceled on  both sides to  give \eqref{ham_ext}.
\end{proof}

\begin{figure}\centering
\includegraphics[scale=.5]{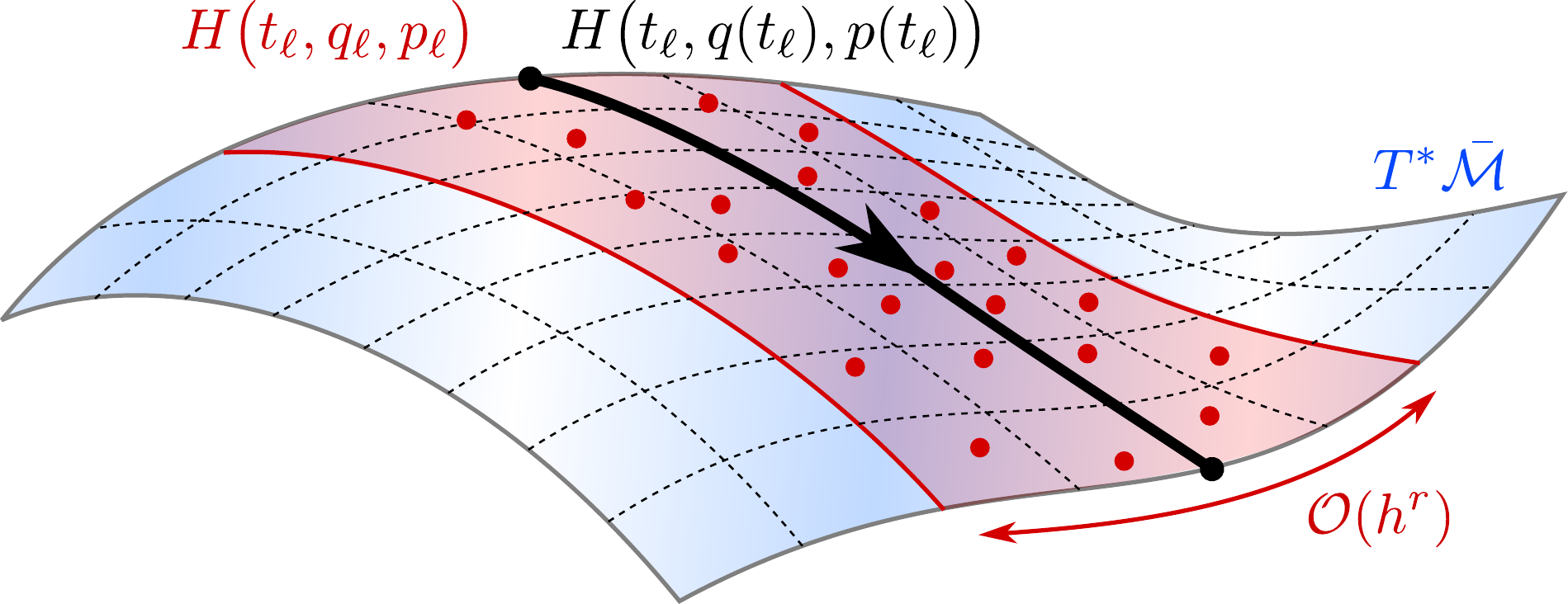}
\caption{\label{surface}
A presymplectic integrator closely preserves the time-dependent Hamiltonian;  Theorem~\ref{preserv_ham_decay}.
The numerical trajectories  lie on the same presymplectic manifold as the nonconservative system and the numerical value $H(t_\m, q_\m, p_\m)$ exhibits a small and bounded error  of size $\bigo(h^r)$ with respect to the true value $H(t_\m, q(t_\m), p(t_\m))$.}
\end{figure}

Theorem~\ref{preserv_ham_decay} is an extension of Theorem~\ref{symp_ham} to nonconservative Hamiltonian systems. The numerical solutions provided by presymplectic integrators are thus within a small and bounded error $\bigo(h^r)$ from the hypersurface  \eqref{surf}; see Fig.~\ref{surface} for an illustration.  Hence, in the case of dissipative systems, whatever convergence or decaying properties the system may have, a presymplectic integrator will closely reproduce its behavior.

In Appendix~%
\ref{constructing_symp} we provide a detailed step-by-step procedure for constructing presymplectic integrators. The procedure is straightforward, essentially involving an application of symplectic integrators with a natural choice for incorporating the time variable.
We now turn to the consequences of Theorem~\ref{preserv_ham_decay} to
dissipative systems.

\section[Preserving rates and stability of dissipative systems]{Preserving rates and stability of \\ dissipative systems} \label{implication_optimization}

Let us consider how presymplectic integrators can be used to construct
``rate-matching'' discretizations.
We consider a standard optimization problem \eqref{optimization}
and study dissipative systems obtained from the
Hamiltonian \eqref{gen_ham}, which we repeat here for convenience:
\be\label{gen_ham2}
H = e^{-\eta_1(t)} T(t, q, p) + e^{\eta_2(t)} f(q)
\ee
where $\eta_1, \eta_2$  are positive and nondecreasing. These functions are responsible for introducing dissipation in the system.  Recall that we assume  the kinetic energy $T$ is Lipschitz continuous.
Given a specific system, suppose that a convergence rate $f(q(t)) - f(q^\star)$
as described in \eqref{gen_rate} is known, where $f(q^\star)$ is a minimum of $f$ in a neighborhood of the initial conditions.  We are now able to prove Theorem~\ref{main_theorem} stated in the introduction as a corollary of Theorem~\ref{preserv_ham_decay}.

\begin{proof}[Proof of Theorem~\ref{main_theorem}]
Replacing \eqref{gen_ham2} into \eqref{ham_ext}, where we now fix $s = t$, yields
\be \label{first_rate}
f(q_\m) \le f(q(t_\m))  + e^{-\eta_1(t_\m) - \eta_2(t_\m)}\left\{ T\big(t_\m, q(t_\m), p(t_\m)\big) - T\big(t_\m, q_\m, p_\m\big) \right\} + K h^r e^{-\eta_2(t_\m)}
\ee
for all $t_\m = \bigo( h^r e^r e^{h_0/h})$ and some constant $K > 0$.
Recall that $q_\m$ and $q(t_\m)$ are the discrete- and continuous-time trajectories, respectively.
Let $L_T$ be the Lipschitz constant of the kinetic energy $T$.  Substituting the global error \eqref{glob_err} into the second term gives
\be \label{fbound}
| f(q_\m) - f(q(t_\m)) | \le  
 h^r e^{-\eta_2(t_\m)}  \left( K + L_T C ( e^{L_{\flowN}t_\m} - 1 ) e^{-\eta_1(t_\m)}  \right).
\ee
We thus conclude that
\be \label{frate2}
f(q_\m) = f(q(t_\m)) + \bigo\left( h^r e^{-\eta_2(t_\m)} \right)
\ee
provided
\be \label{eta_bound}
(e^{L_{\flowN} t} - 1) e^{-\eta_1(t)} < \infty
\ee
for all sufficiently large $t$.
\end{proof}

Therefore, under suitable conditions, presymplectic integrators reproduce the continuous-time rates of convergence of dissipative Hamiltonian systems.  We make several remarks:

\begin{itemize}
\item As mentioned earlier, the assumptions of Theorem~\ref{main_theorem} are  mild;
the kinetic energy being Lipschitz is naturally satisfied in a range of applications, and the Lipschitz condition \eqref{lipschitz_flow} is also satisfied for a large class of integrators. The underlying reason is because $\M$ is a smooth manifold.

\item We  used a conservative global error \eqref{glob_err} to bound the kinetic term in \eqref{first_rate}. A specialized analysis may yield a better bound, which could allow condition \eqref{eta_bound} to be relaxed. However, this would have to be justified by a dedicated backward error analysis.

\item Increasing the order of accuracy $r$  influences \eqref{frate2} in a beneficial way.  However, the error is dominated by $e^{-\eta_2}$ which suggests that in this optimization context using higher-order integrators may not be of significant practical importance.  That said, higher-order methods tend to be more stable so there may be other benefits.

\item Interestingly, the error in \eqref{frate2} is even smaller than the error in approximating the Hamiltonian \eqref{ham_ext}, thanks to the function $\eta_2$. If $\eta_2$ grows just enough, this error may become completely negligible, even for large step sizes and integrators of low order.

\item We recall that the restriction $t_\m = h\m \sim h^r e^r e^{h_0/h}$ may be irrelevant.
To give an idea, with $r=1$, $h_0=1$, and $h=0.1$ we have $\m \sim 6\cdot 10^4$, and with $h=0.01$ we have $\m \sim 10^{43}$. These are sufficiently large iteration
numbers for most practical purposes.

\end{itemize}

Note that Theorem~\ref{main_theorem} is a consequence of
Theorem~\ref{preserv_ham_decay}, which previously was known to hold only for
conservative systems.  The implications are similar to those of the conservative case.  Indeed, the reason why presymplectic integrators are able to closely reproduce the time-dependent Hamiltonian is because---as in the
conservative case---they exactly preserve
a shadow Hamiltonian, despite being
nonconservative.  This is important physically since, in principle,
the Hamiltonian
contains all the information one may wish to extract about the system.
In particular, our result implies that such methods
closely preserve the \emph{energy}  as well---which is not the
Hamiltonian since it depends explicitly on time.
For instance, within the generalized conformal Hamiltonian formalism
discussed in Appendix~%
\ref{generalized_conformal}, the ``physical energy''
and ``physical momenta'' are given by
\be
\mathcal{E}\big(q, p_{\textnormal{mech}}\big) =  e^{-\eta(t)} H(t, q, p),
\qquad
p_{\textnormal{mech}} = e^{-\eta(t)} p.
\ee
Note that, in contradistinction to $H$, the energy  $\mathcal{E}$ \emph{does not
depend explicitly on time}.
Thus, since presymplectic integrators closely preserve $H$, they
also closely preserve $\mathcal{E}$.
This provides another perspective
on the result expressed in Theorem~\ref{main_theorem}. Indeed, if
for a particular system we have that
$\E \ge 0 $ and $\dot{\E} \le 0$, then $\E$ is also a
\emph{Lyapunov function}.
\footnote{As a concrete example, consider \eqref{quad_kin_breg}
below with $\eta_1 = \eta_2 = \eta(t)$.  The physical energy
is given by $\mathcal{E} = (1/2) \dot{q} \cdot M \dot{q} + f(q)$,
and with $p_{\textnormal{mech}} = M \dot{q} = e^{-\eta(t)} p$ we
have $\mathcal{E} =  e^{-\eta(t)}H$. In this case $\mathcal{E} \ge 0$
and $d \mathcal{E}/dt = - \dot{\eta} \, \dot{q} \cdot M \dot{q}  \le 0$
(assuming that $M$ is symmetric and positive semidefinite).}
Once a Lyapunov function is available, one
can deduce stability properties of the system around critical points.
Therefore, since
presymplectic integrators also reproduce the behavior of the Lyapunov
function---through $H$---they
faithfully reproduce the phase portrait and the stability of
critical points.
Moreover, the \emph{decaying rate} of $\E$
will also be numerically imitated and this is another way of seeing
why such methods are ``rate-matching.''
We stress that  these ideas correspond to a dissipative generalization
of what
is known to hold true for
structure-preserving integrators of conservative systems, in which
case $\mathcal{E} = H$ is constant.

\subsection{The choice of damping}
Let us comment on some consequences of the condition \eqref{eta_bound} which can impose some limitations. There are two ways to satisfy it.  The first is if the method converges quickly so that  $e^{L_{\flowN}t_\m - \eta_1(t_\m)}$ remains bounded regardless of $\eta_1$, or equivalently
we enforce $t_\m \le \bar{t}$, where $\bar{t}$ is  a solution to
\be
L_{\flowN} \bar{t} -\eta_1(\bar{t}) \le C
\ee
for some constant $C > 0$.
The second way to satisfy \eqref{eta_bound} is if $\eta_1$ grows fast enough, allowing $t_\m$ to be arbitrarily large.
For instance, setting
\be \label{eta1_large}
\eta_1(t) = \gamma t
\ee
for some constant $\gamma \ge L_{\flowN}$ suffices for all $t \ge 0$.  Curiously, this choice implies constant damping and can be related to the heavy ball method \cite{Polyak:1964}---see also \cite{Franca:2019} for details.\footnote{
With $\eta_1 = \eta_2 = \gamma t$ and $T = \tfrac{1}{2}\| p \|^2$ into \eqref{gen_ham2} the equations of motion
give $\ddot{q} + \gamma \dot{q} = - \nabla f(q)$, which is a nonlinear generalization of the damped harmonic oscillator.
The heavy ball method \cite{Polyak:1964} is actually a structure-preserving---conformal symplectic---discretization of this system \cite{Franca:2019}.}

As an alternative, from \eqref{eta_bound} we have
\be \label{eta1_series}
\left(1+L_{\flowN} t + \tfrac{1}{2} (L_{\flowN} t)^2 + \dotsm \right) e^{-\eta_1(t)} < \infty,
\ee
which is satisfied if $(L_{\flowN} t)^m e^{-\eta_1(t)} < \infty$ for all powers $m$, i.e.,
$\eta_1(t) \sim m \log( L_{\flowN} t)$. Thus one can choose
\be \label{eta1_small}
\eta_1(t) = \gamma \log t
\ee
with a suitable $\gamma > 0$ to ensure that \eqref{eta1_series} holds up to some high order term.
Also curiously, the choice \eqref{eta1_small} is related to the damping of Nesterov's method \cite{Nesterov:1983}
which attains the optimal convergence rate for convex functions $f$.\footnote{
Choosing $\eta_1 = \eta_2 = \gamma \log t$ yields the differential equation
$\ddot{q} + \tfrac{\gamma}{t} \dot{q} = -\nabla f(q)$. Nesterov's method can be seen as a discretization of this system \cite{Candes:2016,Franca:2019}.}

Note that the choice \eqref{eta1_small} can guarantee \eqref{eta1_series} only up to a finite power $m$, thus we are on the edge of violating this condition---in fact we are formally violating it but in practice this may yield a viable algorithm. On the other hand,
the choice \eqref{eta1_large} may be overkill. It may be reasonable to consider an intermediate alternative:
\be \label{damp_between}
\eta_1(t) = \gamma_1 \log t + \gamma_2 t^{\delta}
\ee
for constants $0 < \delta \le 1$ and $\gamma_1,\gamma_2 > 0$. With this choice the condition on the $m$th term in the series \eqref{eta1_series}
becomes $L_{\flowN}^m  t^{m - \gamma_1}  e^{-\gamma_2  t^{\delta}} < \infty$, which can be guaranteed even if $\gamma_1 < m$ with suitable choices of $\gamma_2$ and $\delta$.

Summing up, presymplectic integrators are able to match the continuous-time
rates provided
the condition \eqref{eta_bound} is satisfied, which involves an appropriate choice of damping. If the system is overdamped this condition is more likely to hold, however the system would tend to be slower. On the other hand, if the damping is weak the system tends to be fast, but at the cost of violating \eqref{eta_bound}. We thus see the tradeoff is delicate and is  problem-dependent in general, given that $L_\flowN$ in \eqref{eta_bound} is related to the Lipschitz constant of the original vector field which depends on the potential $f(q)$.

\section{Bregman dynamics} \label{bregman_sec}

Given that the Bregman Hamiltonian \cite{Wibisono:2016} provides a general framework for deriving continuous-time optimization procedures, we consider this case in some detail. The Hamiltonian has the form
\be \label{bregman_ham}
  H = e^{\alpha + \gamma} \left\{ D_{h^\star}\!\left( \nabla h(q) + e^{-\gamma} p, \nabla h(q) \right)
  + e^{\beta} f(q) \right\},
\ee
where $\alpha$, $\beta$, and $\gamma$ are all functions of time $t$, and  required to satisfy the following scaling conditions:
\be \label{scaling}
  \dfrac{d \beta}{dt} \le e^{\alpha}, \qquad \dfrac{d \gamma}{dt} = e^{\alpha}.
\ee
The kinetic energy in \eqref{bregman_ham} is given in terms of the Bregman divergence,
\be \label{breg_div}
  D_h(y,x) \equiv h(y) -h (x) -  \big \langle \nabla h(x), y-x \big\rangle,
\ee
which is nonnegative for a given convex function $h : \M \to \mathbb{R}$. Recalling the definition of the convex
dual, $h^\star$ of $h$, defined by the Legendre-Fenchel transformation
\be \label{breg_sup}
  h^\star(p) \equiv \sup_{v} \left\{ \langle p, v \rangle - h(v) \right\},
\ee
one can show that Hamilton's equations are given by%
\footnote{Equivalently, one can write the second-order differential equation
$$
\ddot{q} + \big( e^{\alpha} - \dot{\alpha} \big) \dot{q} +
e^{2\alpha + \beta}\left[ \nabla^2 h \left(q + e^{-\alpha} \dot{q}\right)  \right]^{-1}
\nabla f(q) = 0.
$$
We see that $\alpha$ basically controls the damping, while $2\alpha+\beta$ increases the strength of the force $-\nabla f$. Both play a major role in the stability of the system, and they are not independent since $\dot{\beta} \le e^{\alpha}$.
}
\begin{subequations} \label{breg_motion}
\begin{align}
\dot{q} &= e^{\alpha}\left\{ \nabla h^*\left(  \nabla h(q)  + e^{-\gamma }p \right) - q   \right\}, \\
\dot{p} &= - e^{\alpha+\gamma} \nabla^2 h(q) \left\{  \nabla h^*\left(\nabla h (q) + e^{-\gamma} p\right) -q\right\}
+ e^{\alpha} p - e^{\alpha + \beta + \gamma} \nabla f(q).
\end{align}
\end{subequations}
In the case where  $f$ is convex,  the convergence rate of this system is given by \cite{Wibisono:2016}
\be \label{rate_breg}
f(q(t)) - f(q^\star) = \bigo\left( e^{-\beta(t)} \right).
\ee
Moreover, for a given $\alpha$, the optimal rate is obtained with $\dot{\beta}  = e^{\alpha}$, thus
$\beta = \gamma + C$, for some constant $C$, and both are determined  in terms of $\alpha$. We thus have:
\be \label{optimal_rate}
f(q(t)) - f(q^\star) = \bigo\left( e^{-\beta(t)} \right), \qquad  \beta(t) = \int^t e^{\alpha(t')} dt', \qquad \gamma(t) = \beta(t) + C.
\ee
A choice considered by \cite{Wibisono:2016} is
\be \label{scaling2}
\alpha = \log c - \log t, \qquad \beta = c \log t + C, \qquad  \gamma = c \log t ,
\ee
with $c > 0$,
whereby the convergence rate \eqref{optimal_rate} becomes the polynomial $\bigo\left( t^{-c} \right)$.
Another possibility is
\be\label{scaling1}
\alpha = \log c, \qquad \beta = c t, \qquad \gamma = c t + C  ,
\ee
leading to an exponential rate of $\bigo(e^{-c t})$.

\subsection{Separable case}

To consider a case where the Bregman Hamiltonian is separable, let us start with the choice
$h(x) = \tfrac{1}{2} x \cdot M x $
with a symmetric positive semidefinite matrix $M$ so that
$h^\star(x) = \tfrac{1}{2} x \cdot M^{-1} x$ and the Bregman divergence
is given by
$D_{h^*}(y,x) = \tfrac{1}{2} (y-x) \cdot  M^{-1}(y-x)$.
In this case the  Hamiltonian \eqref{bregman_ham} takes the following form:
\be\label{quad_kin_breg}
H = \dfrac{1}{2} e^{- \eta_1(t) }  p_i   (M^{-1})^{ij} p_j  + e^{ \eta_2(t) } f(q)
\ee
where we defined
\be \label{damp_func}
\eta_1 \equiv \gamma -\alpha , \qquad \eta_2 \equiv \alpha + \beta + \gamma.
\ee
The quadratic kinetic energy is standard in classical mechanics and we see that $M$ plays the role of a mass matrix.
We thus have the equations of motion
\be \label{ham_class}
\dot{q} = e^{-\eta_1(t)} M^{-1} p, \qquad
\dot{p} = - e^{\eta_2(t)} \nabla f(q).
\ee
One can now use any presymplectic integrator  to simulate such a dissipative system---see Appendix~%
\ref{constructing_symp} for several choices.
Because the Hamiltonian \eqref{quad_kin_breg} is separable, standard approaches yield explicit methods which are convenient in practice.
Assuming that we use a presymplectic integrator of order $r$,
Theorem~\ref{main_theorem} tells us that the continuous-time rate of
convergence will be preserved:
\be \label{quad_rate}
\begin{split}
f(q_\m) - f^\star &=
\bigo\left( e^{-\beta(t_\m)}\right)
+\bigo\left( h^r e^{-\alpha(t_\m)-2\beta(t_\m)}\right) \\
&\sim e^{-\beta(t_\m)}\left(1+h^r e^{-\alpha(t_\m)-\beta(t_\m)}\right) \\
&\sim e^{-\beta(t_\m)}
\end{split}
\ee
provided $\eta_1$ as defined in \eqref{damp_func} obeys the
condition \eqref{eta_bound}.

We present two explicit methods. Using the presymplectic Euler method
given by \eqref{presymp_euler}, which has order $r=1$, we obtain
the following updates:
\be \label{quad_seu}
\begin{split}
p_{\m+1} &= p_{\m} - h e^{\eta_2(t_{\m})} \nabla f (q_\m), \\
t_{\m+1} &= t_{\m} + h, \\
q_{\m+1} &= q_{\m} + h e^{-\eta_1(t_{\m})} M^{-1} p_{\m+1}.
\end{split}
\ee
This algorithm requires only one gradient computation per iteration.
In a similar way, using the presymplectic leapfrog method
\eqref{presymp_leapfrog}, which
has order $r=2$, we have:
\be \label{quad_leap}
\begin{split}
p_{\m+1/2} &= p_{\m} - (h/2) e^{\eta_2(t_\m)}\nabla f(q_{\m}), \\
t_{\m+1} &= t_{\m} + h, \\
q_{\m+1} &= q_{\m} + (h/2) \left( e^{-\eta_1(t_\m)} + e^{-\eta_1(t_{\m+1})}
    \right) M^{-1} p_{\m+1/2}, \\
p_{\m+1} &= p_{\m+1/2} - (h/2) e^{\eta_2(t_{\m+1})} \nabla f(q_{\m+1}).
\end{split}
\ee
This method also only requires one gradient computation per iteration---the
last one can be reused in the subsequent iteration---but it is more accurate
and stable than \eqref{quad_seu}.
Note also that similarly to \eqref{quad_seu} and \eqref{quad_leap} we can
consider \eqref{presymp_euler2}
and \eqref{presymp_leapfrog2}, respectively, which yield adjoint methods
to the ones above.
One can now choose any suitable scaling functions in \eqref{damp_func} and
substitute into either \eqref{quad_seu} or \eqref{quad_leap} to obtain a
specific optimization algorithm.

\subsection{Nonseparable case} \label{explicit_bregman}

We turn to the case of a nonseparable Hamiltonian, where the integrators considered thus far yield implicit updates that require solving nonlinear equations. This not only increases the computational burden but can affect the numerical stability. In particular, for the Bregman Hamiltonian \eqref{bregman_ham}, we need a construction suited to nonseparable cases.   In order to obtain \emph{explicit} methods,
motivated by \cite{Tao:2016} we introduce extra degrees of freedom. Thus,
given a  Hamiltonian $H(t,q,p)$, we double its degrees of freedom  by introducing an augmented Hamiltonian:
\be\label{tao_ham}
\bar{H}(t,q,p,\bar{t}, \bar{q},\bar{p}) \equiv H(t, q,\bar{p}) + H(\bar{t}, \bar{q}, p) +
\dfrac{\xi}{2} \left( \| q-\bar{q} \|^2 + \| p - \bar{p} \|^2 \right).
\ee
Here $\xi > 0$ is a coupling constant that controls the strength of the last term,
which forces $q = \bar{q}$ and $p = \bar{p}$.\footnote{It is not necessary to introduce the same term
for $t$ and $\bar{t}$ since they will  be equal thanks to  \eqref{q0}.
The constant $\xi$ has to be carefully tuned in practice;
see \cite{Tao:2016} for details.}
The presymplectic structure of this system is now
\be
\Omega = -dH\wedge dt + dp_j\wedge dq^j - dH \wedge d\bar{t} +  d\bar{p}_j \wedge d\bar{q}^j  . 
\ee
Hamilton's equations, obtained from \eqref{tao_ham}, will preserve $\Omega$.  The equations of motion are equivalent to those of the original system when
$q = \bar{q}$, $p = \bar{p}$ and $t = \bar{t}$.
We thus propose the numerical maps
\begin{equation} \label{tao_map1}
%
\flowN_{h}^{A}\begin{pmatrix} t \\ q \\ p \\ \bar{t} \\ \bar{q} \\ \bar{p} \end{pmatrix} =
\begin{pmatrix}
t \\ q \\ p - h \nabla_q H(t, q, \bar{p}) \\ \bar{t} + h \\ \bar{q} + h \nabla_{\bar{p}} H(t, q, \bar{p}) \\ \bar{p}
\end{pmatrix}, \qquad
\flowN_{h}^B \begin{pmatrix} t \\ q \\ p \\ \bar{t} \\ \bar{q} \\ \bar{p} \end{pmatrix} =
\begin{pmatrix}
t + h \\ q + h \nabla_p H(\bar{t}, \bar{q}, p) \\ p \\ \bar{t} \\ \bar{q}  \\ \bar{p} - h \nabla_{\bar{q}} H(\bar{t}, \bar{q}, p)
\end{pmatrix} ,
\end{equation}
and
\begin{equation} \label{tao_map2}
\flowN_{h}^C \begin{pmatrix} t \\ q \\ p \\ \bar{t} \\ \bar{q} \\ \bar{p} \end{pmatrix} =
\dfrac{1}{2}\begin{pmatrix}
2 t \\
q+\bar{q} + \cos(2\xi h)(q-\bar{q}) + \sin(2\xi h)(p-\bar{p}) \\
p+\bar{p} - \sin(2\xi h)(q-\bar{q}) + \cos(2\xi h)(p-\bar{p}) \\
2 \bar{t} \\
q+\bar{q} - \cos(2\xi h)(q-\bar{q}) - \sin(2\xi h)(p-\bar{p}) \\
p+\bar{p} + \sin(2\xi h)(q-\bar{q}) - \cos(2\xi h)(p-\bar{p}) \\
\end{pmatrix} .
\end{equation}
A presymplectic integrator for any (nonseparable) time-dependent Hamiltonian system can then be constructed by composing these maps. For instance, the Strang composition
given by
\be \label{tao_flow}
\flowN_{h/2}^{A} \circ \flowN_{h/2}^B \circ \flowN_{h}^C \circ \flowN_{h/2}^B \circ \flowN_{h/2}^A
\ee
is known to generate a method of order $r=2$---this is the same composition as the leapfrog.
Note that the maps \eqref{tao_map1}--\eqref{tao_map2} are completely explicit in all variables and this approach works for any time-dependent Hamiltonian.

In particular, for the Bregman Hamiltonian \eqref{bregman_ham} it suffices
to substitute the gradients
$\dot{q} = \nabla_p H$ and $\dot{p} = - \nabla_q H$ from  \eqref{breg_motion}
into \eqref{tao_map1}--\eqref{tao_map2}, followed by the composition \eqref{tao_flow}.
This  yields an explicit, second-order
integrator for the general Bregman dynamics.
Thanks to Theorem~\ref{main_theorem}, this generates an optimization algorithm that may closely preserve
the continuous rate of convergence \eqref{rate_breg} for suitable functions $\alpha$, $\beta$ and $\gamma$.

\section{Numerical experiments} \label{numerical}

The purpose of this section is twofold. First, we aim to investigate
the claim of Theorem~\ref{preserv_ham_decay}---our main
result regarding structure-preserving integrators for nonconservative
Hamiltonian systems---numerically.  Second, we illustrate that our approach can
form the basis for constructing reliable optimization methods based on
dissipative Hamiltonian dynamics according to Theorem~\ref{main_theorem}.

\subsection{Error growth in the Hamiltonian}

In order to verify Theorem~\ref{preserv_ham_decay} we need to
consider cases where an exact analytical solution is
available.
Thus, consider the Hamiltonian \eqref{quad_kin_breg} with $M=I$
for simplicity and with the one-dimensional potential
$f(q) = q^2 / 2$ (higher-dimensional quadratic functions can be treated
similarly by decoupling the degrees of freedom into normal modes).
First, choose $\eta_1 = \eta_2 = \gamma t$
so Hamilton's equations become
the damped harmonic oscillator:
\be
\ddot{q} + \gamma \dot{q} + q = 0.
\ee
The exact solution of the initial value
problem with $q(0) = q_0$ and $p(0)=0$ is given by
\be \label{sol_const}
q(t) = q_0 e^{-t \gamma/2}\left( \cos\left( \dfrac{\omega t}{  2}\right) +
 \dfrac{\gamma}{\omega} \sin \left(\dfrac{\omega t}{2}\right)  \right), \qquad
p(t) = -\dfrac{2 q_0}{\omega} e^{t \gamma / 2} \sin\left(\dfrac{\omega t}{2}
\right),
\ee
where $\omega \equiv \sqrt{4 - \gamma^2}$.
A plot of these functions as well as the Hamiltonian
is in Fig.~\ref{exact_sol_const}.

\begin{figure}
\centering
\includegraphics[width=0.95\textwidth]{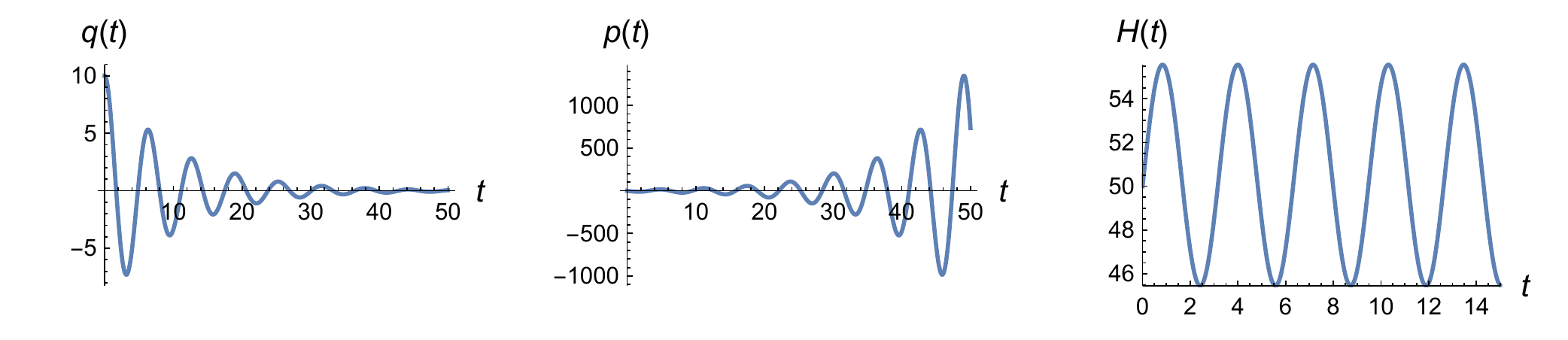}
\caption{Exact solution \eqref{sol_const} with
$\gamma = 0.2$ and $q_0=10$.
We replaced the actual trajectory and momentum into
the Hamiltonian, which varies with time and is not conserved.
\label{exact_sol_const}
}
\end{figure}

We numerically integrate the equations of motion using
the presymplectic Euler method \eqref{quad_seu}, the
presymplectic leapfrog \eqref{quad_leap},
and a method
of order $r=4$ based on the Suzuki-Yoshida construction
\eqref{yoshida_comp} with \eqref{quad_leap} as  the base method.
All of these integrators obtain accurate estimates of $q(t)$ and
$p(t)$. In Fig.~\ref{quad_const_num} (left) we show
the error in the numerical estimate of the Hamiltonian
over a sufficiently large simulation time. Note that this error
remains bounded.  To verify \eqref{ham_ext} in more detail,
we fix a simulation time and vary the step size to compute the maximum
error
\be \label{max_error}
\max_\m | H(t_\m, q(t_\m), p(t_\m)) - H(t_\m, q_\m, p_\m) |
\ee
over the entire history of the system; each simulation
uses a fixed step size $h$.
We report the results for this problem in
Fig.~\ref{quad_const_num} (right).
As we can see, the relation \eqref{ham_ext} is satisfied precisely.

\begin{figure}[t]
\centering
\includegraphics[width=.45\textwidth]{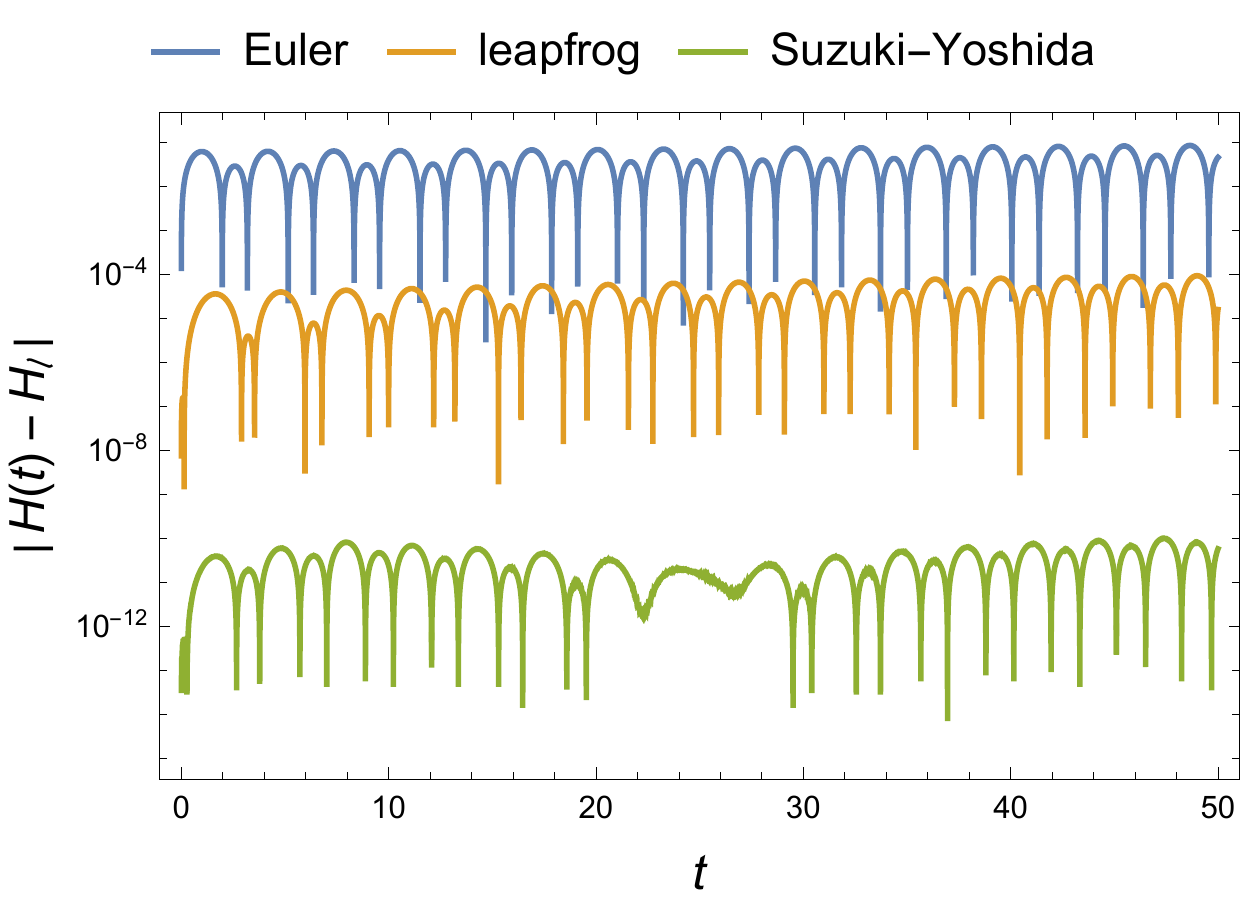}\qquad%
\includegraphics[width=.45\textwidth]{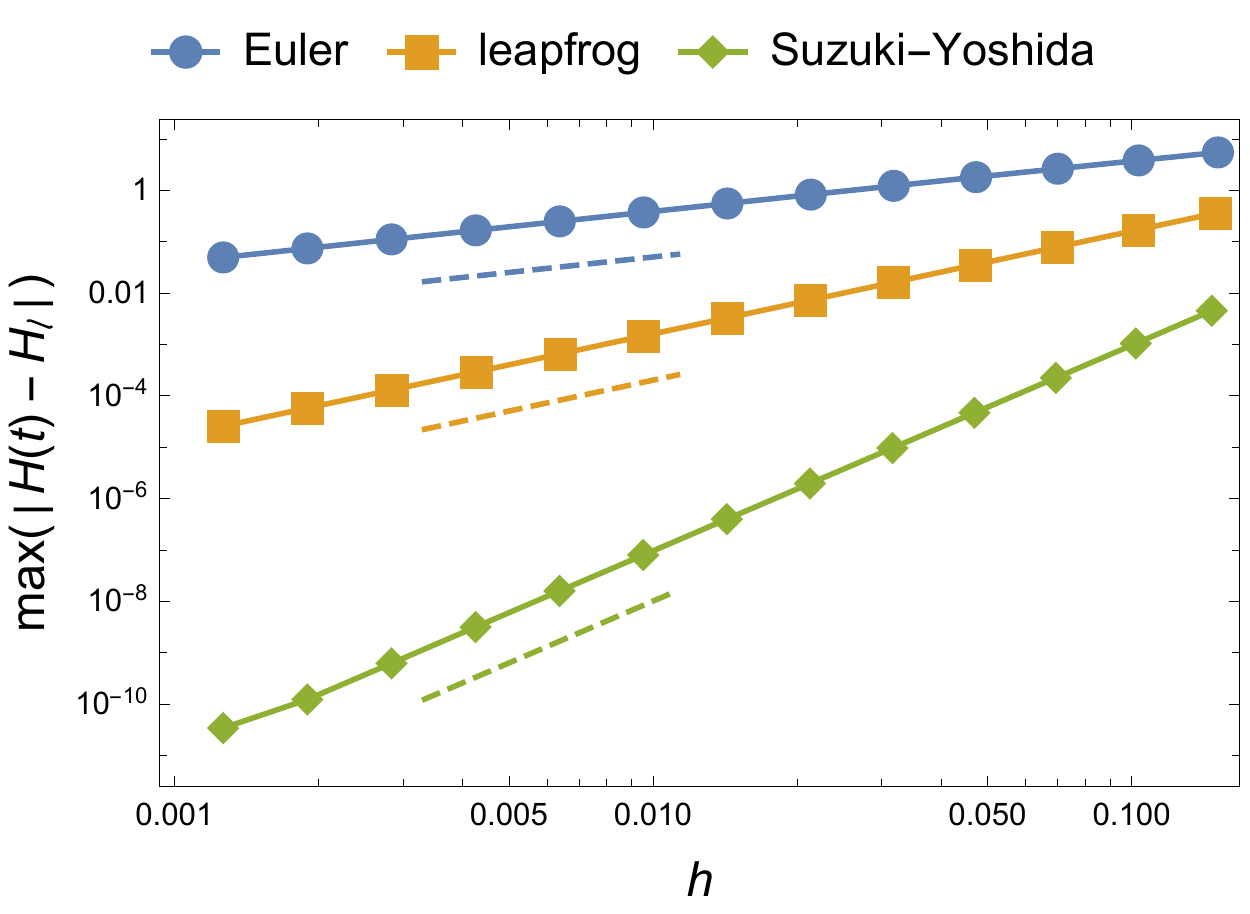}%
\caption{Error in the nonconserved Hamiltonian using the exact
solution \eqref{sol_const}
and numerical estimates with presymplectic integrators.
\emph{Left:} We show $|H(t_\m) - H_\m|$ (in log scale)
for  $t_{\textnormal{max}} = 50$ and
$\m_{\textnormal{max}}= 3 \times 10^4$ iterations, thus
a step size $h = t_{\textnormal{max}}/\m_{\textnormal{max}}$.
We choose $\gamma = 0.2$ as in
Fig.~\ref{exact_sol_const}.
We use the presymplectic Euler \eqref{quad_seu}, the leapfrog
\eqref{quad_leap}, and
the Suzuki-Yoshida \eqref{yoshida_comp} with the latter
as the base method. The error remains bounded over time
even though the Hamiltonian is not conserved. \emph{Right:} Under
the same setting, we
perform several simulations during a time interval $t \in [0,10]$
with different step sizes $h$ and calculate \eqref{max_error}.
The dashed lines were generated independently to verify
the errors $\bigo(h)$, $\bigo(h^2)$ and $\bigo(h^4)$, respectively.
\label{quad_const_num}
}
\end{figure}

As a second example, consider the one-dimensional potential $f(q) = q^2/2$
but now with damping functions $\eta_1 = \eta_2 = \gamma \log(t+1)$.
We thus have an harmonic oscillator with decaying damping:
\be \label{decaying_oscillator}
\ddot{q} + \dfrac{\gamma}{t+1} \dot{q} + q = 0.
\ee
The exact solution of this system with initial conditions $q(0) = q_0$ and
$p(0)=0$ is
\be \label{sol_dec}
\begin{split}
q(t) &= \dfrac{q_0 \pi}{2}  (t+1)^{-\alpha_-} \left(
J_{\alpha_+}(1)
Y_{\alpha_-}(t+1) -
Y_{\alpha_+}(1)
J_{\alpha_-}(t+1)
\right), \\
p(t) &= \dfrac{q_0 \pi}{2} (t+1)^{\alpha_+}\left(
Y_{\alpha_+}(1)J_{\alpha_+}(t+1) -
J_{\alpha_+}(1)Y_{\alpha_+}(t+1)
\right),
\end{split}
\ee
where $\alpha_\pm \equiv (\gamma\pm 1)/2$ and
$J_\alpha, Y_\alpha$ denote  Bessel functions of the first and second
kinds, respectively.
A plot of these functions together with the Hamiltonian is
in Fig.~\ref{exact_sol_dec}.

\begin{figure}[t]
\centering
\includegraphics[width=0.95\textwidth]{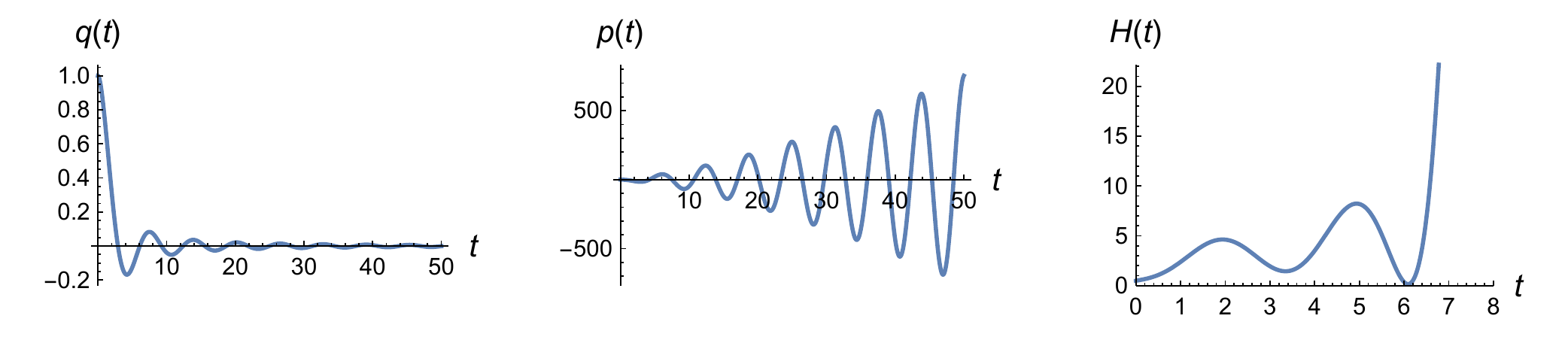}
\caption{Exact solution \eqref{sol_dec} with
$\gamma = 3$ and $q_0=1$.
We replaced the position and momentum into
the Hamiltonian which oscillates and grows very fast with $t$.
\label{exact_sol_dec}
}
\end{figure}

\begin{figure}[t!]
\centering
\includegraphics[width=.45\textwidth]{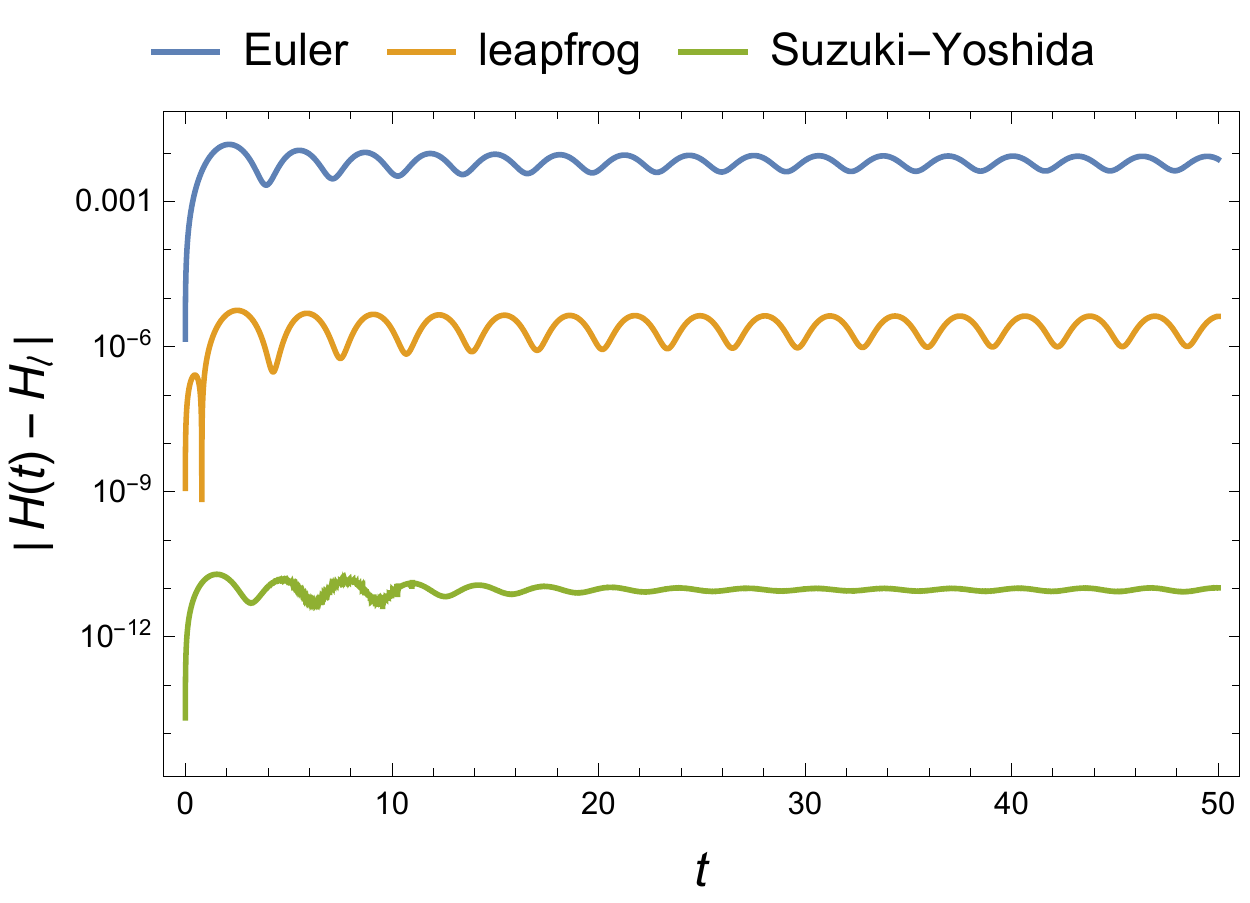}\qquad%
\includegraphics[width=.45\textwidth]{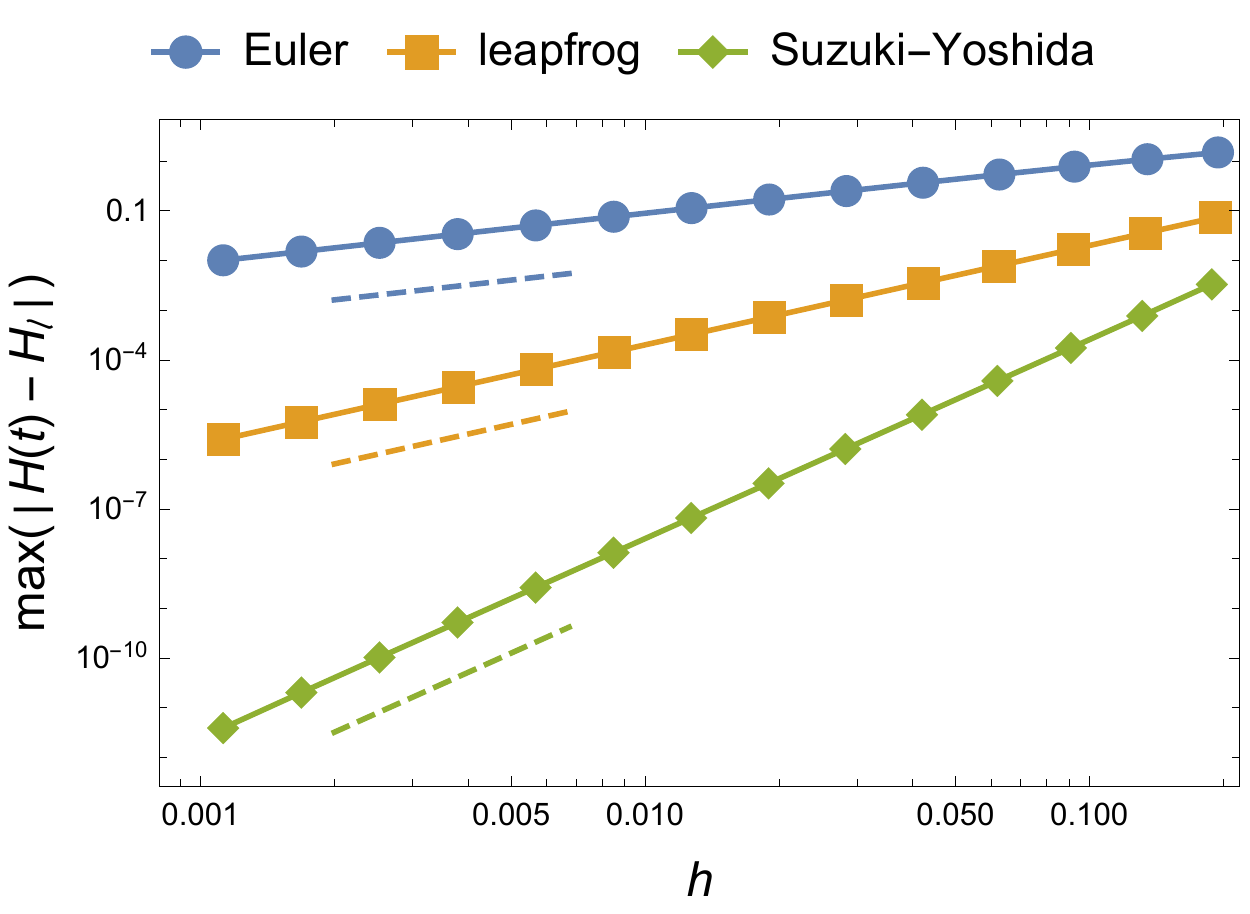}%
\caption{Error in the nonconserved Hamiltonian.
\emph{Left:} We show $|H(t_\m) - H_\m|$ (in log
scale) for $t_{\textnormal{max}} = 50$ and
$\m_{\textnormal{max}}= 3 \times 10^4$
(step size $h = t_{\textnormal{max}}/\m_{\textnormal{max}}$).
We choose $\gamma = 3$ as in
Fig.~\ref{exact_sol_dec}.
\emph{Right:} For each choice of step size, we compute
\eqref{max_error} over $t \in [0, 10]$.
The dashed lines are independent plots to verify
the errors $\bigo(h)$, $\bigo(h^2)$, and $\bigo(h^4)$, respectively,
as predicted in \eqref{ham_ext}.
\label{quad_dec_num}
}
\vspace*{1em}
\includegraphics[width=0.45\textwidth]{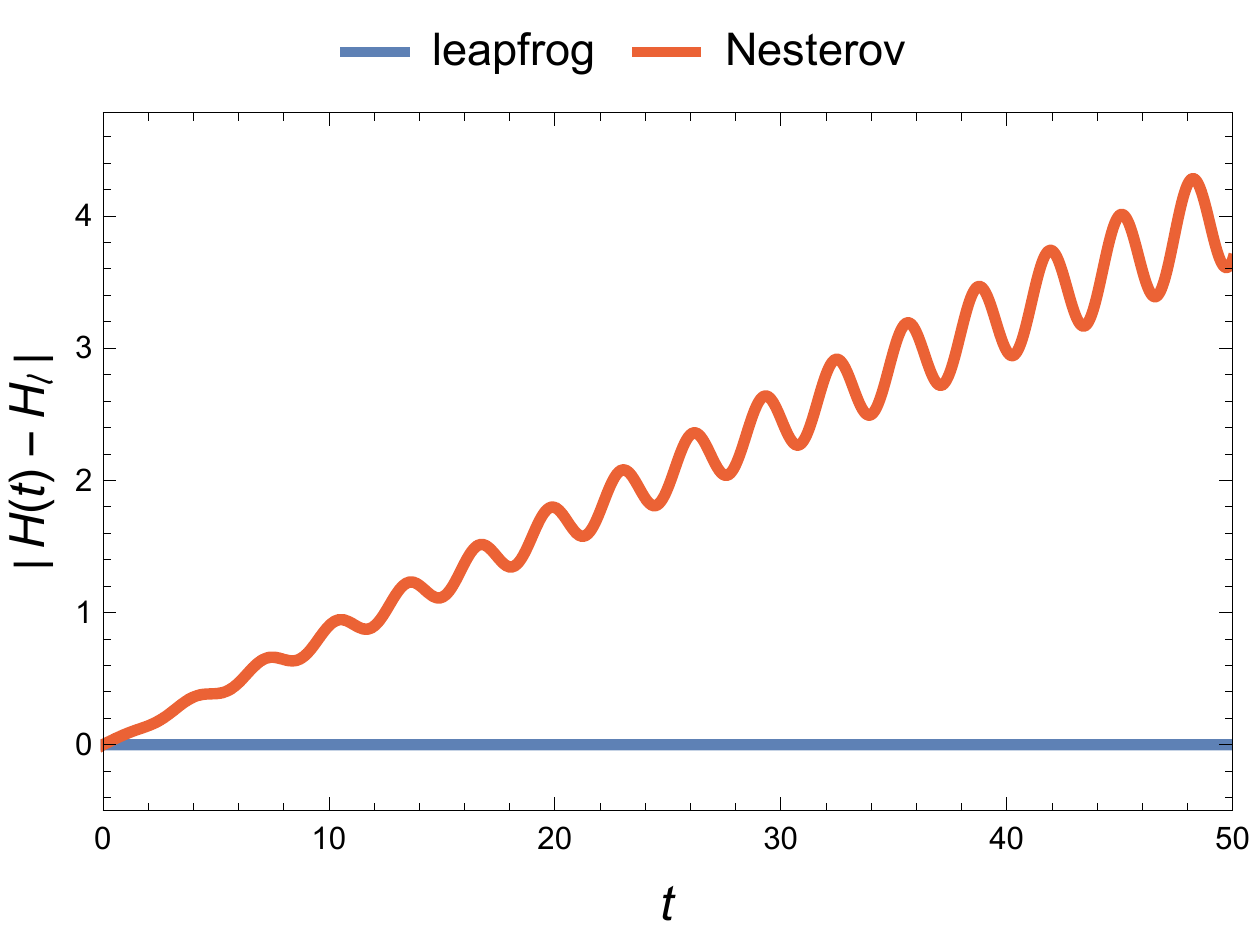}\qquad%
\includegraphics[width=0.45\textwidth]{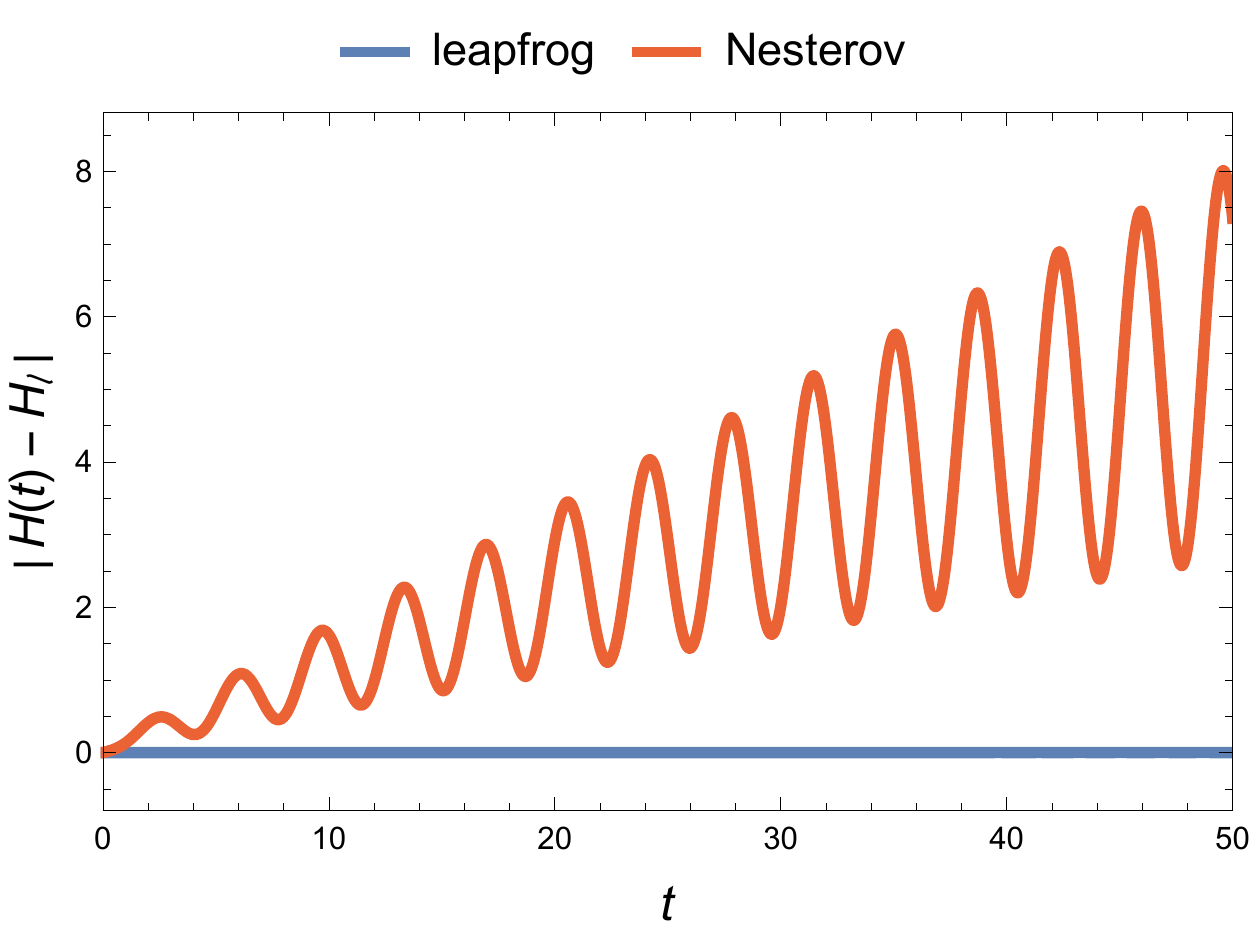}%
\caption{Contrary to our structure-preserving (presymplectic) integrators,
Nesterov's method has an unbounded error in the Hamiltonian.
\emph{Left:} We use exactly the same setting as in Fig.~\ref{quad_const_num}.
\emph{Right:} We choose $\gamma=-1$ so that we ``inject'' energy into
the system; note that a presymplectic method closely preserves---up to
a bounded error $\order(h^r)$---general time-dependent Hamiltonians
(not only for dissipative systems) in agreement with Theorem~\ref{preserv_ham_decay}.
\label{quad_nest}
}
\end{figure}

Using the same integrators as in the previous case, we
obtain the results shown in Fig.~\ref{quad_dec_num}.
Again, relation \eqref{ham_ext} is verified perfectly.
As expected, the leapfrog \eqref{quad_leap} is more stable
than the Euler method \eqref{quad_seu}, since it has order $r=2$; however,
the method based on Suzuki-Yoshida is even more stable and
accurate than both (as expected since it has order $r=4$).
Let us mention that we also implemented
their adjoint integrators as well, i.e., methods
based on \eqref{presymp_euler2}, \eqref{presymp_leapfrog2} and
\eqref{yoshida_comp}; the results
were essentially the same.

For comparison, let us consider Nesterov's method---a well-known
optimization algorithm in the literature---that corresponds to a
discretization of the Hamiltonian system \eqref{ham_class}
with $M=I$.  Nesterov's method
is provably a first-order integrator however it is not structure-preserving~\cite{Franca:2019}.
By considering the same setting as that of Fig.~\ref{quad_const_num},
we obtain the results shown in Fig.~\ref{quad_nest} (left).
Note how Nesterov's has a \emph{growing error} in the Hamiltonian, in contrast
to the presymplectic leapfrog which we include as a reference.
Moreover, in Fig.~\ref{quad_nest} (right) we
introduce \emph{excitation} in the system by choosing instead $\gamma = -1$.
Note how the presymplectic leapfrog still closely reproduces the
Hamiltonian, i.e., up to a bounded error $\order(h^r)$, in
agreement with Theorem~\ref{preserv_ham_decay} which holds not only for
dissipative systems but also for general nonconservative systems.
We also conducted numerical experiments in the case of
system \eqref{decaying_oscillator} with similar results.

\subsection{Quadratic programming}

Consider minimizing
a random quadratic function in the unconstrained case:
\begin{equation} \label{random_quad}
\min_q \left\{ f(q) \equiv \dfrac{1}{2} q \cdot M q \right\},
\end{equation}
where $M = \tfrac{1}{n} A^T A$, with $A$ an $r \times n$ matrix,
$r / n \to y$ for some $0< y \le 1$,
and where the entries $A^{ij}$ are sampled from a standard normal distribution---we
set $n=1000$ and $y = 0.8$ in our example.
Thus $M$ has rank $r$ and its eigenvalues are distributed according to
the Marchenko-Pastur law.
We solve problem \eqref{random_quad} with the presymplectic leapfrog
method \eqref{quad_leap} and compare it with
Nesterov's accelerated method.
We consider a constant damping
$\eta_1 = \eta_2 = \gamma t$ for both methods.

\begin{figure}[t]
\centering
\includegraphics[width=0.5\textwidth]{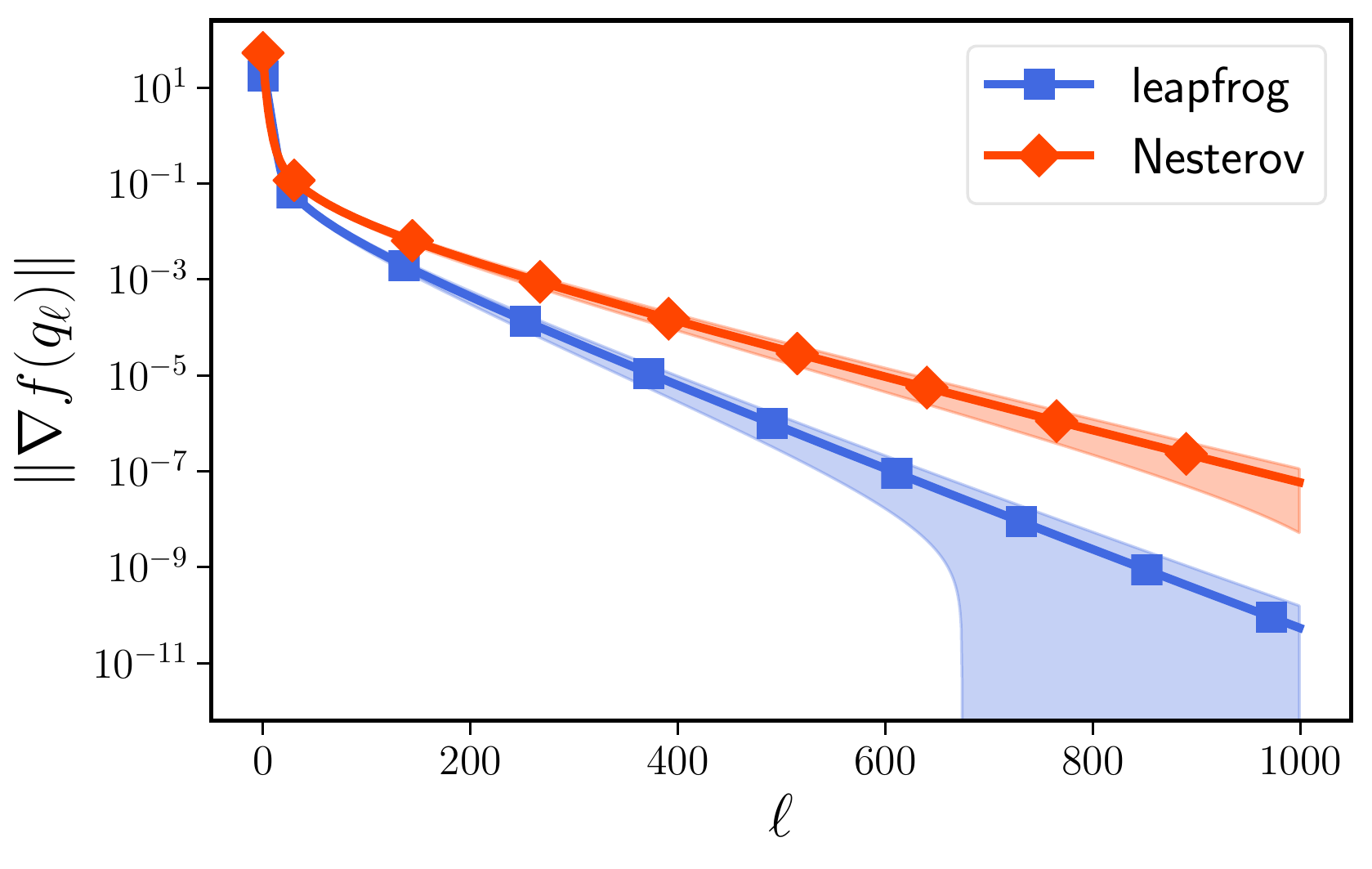}
\caption{We perform 50 Monte Carlo runs when minimizing
\eqref{random_quad} with presymplectic leapfrog \eqref{quad_leap} and
Nesterov's method. We choose $\gamma = 0.7$ for both and step size
$h=0.9$ for leapfrog whereas $h=0.5$ for Nesterov---these were
the largest values such that these
methods converged on all trials.
Solid lines are the mean of $\| \nabla f\|$ and
shaded areas $\pm \sigma$ (standard deviation).
\label{random_quad_fig}
}
\end{figure}

\begin{figure}[t]
\centering
\includegraphics[width=.7\textwidth]{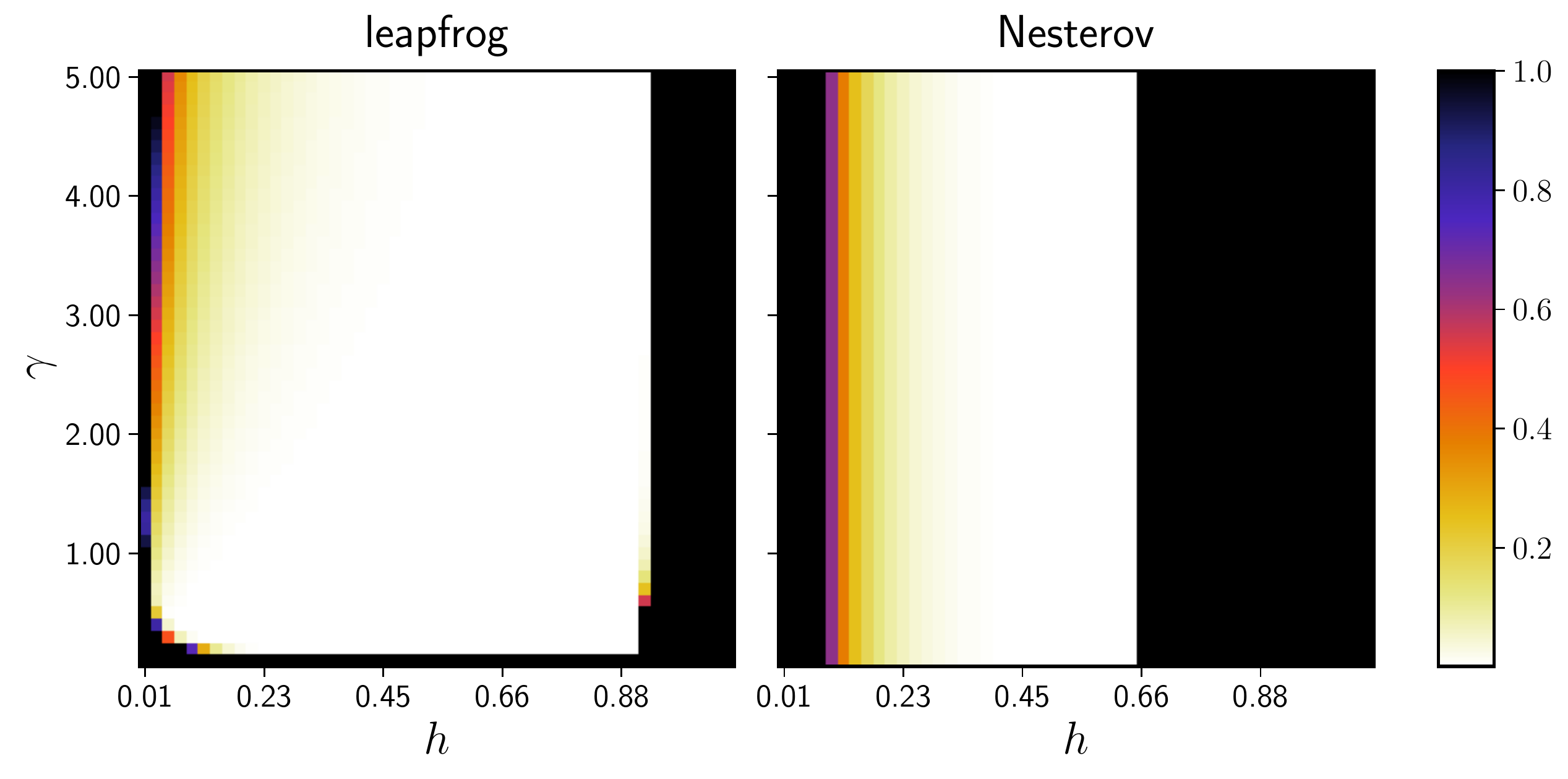}
\caption{Phase diagram of $\| \nabla f\|$ in terms of $(\gamma, h)$.
Dark color means large $\| \nabla f \|$ so that the algorithm diverged
(we truncated larger values to 1 for visualization purposes).
Note the wider area in light color for presymplectic leapfrog
which illustrates its improved stability.
\label{phase_fig}
}
\end{figure}

We thus generate 50 random functions of the kind
\eqref{random_quad} and use these
algorithms with appropriate choice of parameters $(\gamma, h)$.
Their convergence rates are illustrated in
Fig.~\ref{random_quad_fig}
where we plot the mean of $\| \nabla f(q_\m)\|$ against the iteration number
$\m$ (the shaded areas correspond to $\pm$ one standard deviation).
Note that in this case the dissipative version of the leapfrog  was faster
due to a larger choice of step size;
we stress that both methods
have similar performance, the difference is that structure-preserving
discretizations tend to be more stable and thus may accept larger
step sizes for sufficiently well-behaved problems.
To verify the stability of these methods more closely,
in Fig.~\ref{phase_fig} we also show
a $\gamma$-$h$ phase diagram of $\| \nabla f\|$ using a single sample
of the random function \eqref{random_quad};
i.e., for each choice $(\gamma, h)$
we run both algorithms for a certain maximum number of
iterations ($\m_{\textnormal{max}} = 800$) or
until a small tolerance
on $\| \nabla f\|$ is attained (we choose $10^{-3}$).
A light color means a small value of $\| \nabla f\|$ such that the algorithm
converges successfully (the lighter
the smallest value of the gradient), while a
dark color means larger values of $\| \nabla f\|$
(black means complete failure).
Note how the method based on the leapfrog is indeed
more stable since it admits a wider light color region; i.e., it converges
successfully under a wider range of step sizes.
By a closer inspection of this figure we can also
see  that Nesterov's method introduces some spurious damping since
it converges in some cases even when $\gamma = 0$.

\subsection{Learning the Ising model with Boltzmann machines} \label{rbm_sec}

To further illustrate the feasibility of our
approach to optimization, we consider a more realistic problem.
Consider the two-dimensional (ferromagnetic)
Ising model whose Hamiltonian  is given by
\be \label{ising}
H \equiv - \sum_{\langle i j \rangle} \sigma_i \sigma_j ,
\ee
where the spin variables take values $\sigma_i = \pm 1$.
Onsager (1944) solved this system
analytically. This system undergoes a second-order phase transition at
the critical temperature $T_c = 2/\log(1+\sqrt{2})  \approx 2.27$.
The canonical partition function is
\be
Z \equiv \sum_{\{ \sigma_i \}} e^{-\beta H},
\ee
with inverse temperature $\beta \equiv 1/T$, from which any thermodynamic property of the
system can be computed, including the average energy and heat capacity:
\be \label{observables}
\langle E \rangle = -\dfrac{\partial \log Z}{\partial \beta}, \qquad
\langle C \rangle = \dfrac{\partial \langle E \rangle}{\partial T} =
\dfrac{\langle E^2 \rangle - \langle E \rangle^2}{ T^2}.
\ee

We want to ``learn'' the canonical
distribution $\rho = e^{-\beta H} / Z$ using a generative
model known as \emph{restricted Boltzmann machine} (RBM)---such an
approach has recently attracted significant interest in physics
\cite{Torlai:2016,Carleo:2017,Melko:2018,Morningstar:2018}.
Briefly, an RBM is a neural network with two layers, the so-called \emph{visible} layer which
has a number of neurons equal to the dimensionality of the input data vector,
and the \emph{hidden} layer which can have an arbitrary number of neurons.
These two layers are connected with each other through parameters (``weights'')  that we collectively denote by $q$
(neurons in the same layer are not connected).
Given training data---e.g., admissible
states of the Ising model at a given temperature $T$---an RBM is trained by minimizing a ``loss function'' $f(q)$ which can be seen
as an energy function that depends on the weights of the network.
The gradient of this loss function can be approximated
via a technique called contrastive divergence \cite{Hinton:2002}, which
involves running a Markov chain for $k$ steps.  In practice, a
simple \emph{gradient descent} method
is often used in RBM's~\cite{Torlai:2016,Morningstar:2018}.
After the RBM is trained, it provides a statistical representation of the probability distribution of the data, thus allowing  us
to sample the canonical distribution and
estimate thermodynamic quantities such as \eqref{observables}.
Naturally, to train the RBM we need generate data, which can done
by simulating the Ising model with a Monte Carlo method.

\begin{figure}
\centering
\includegraphics[width=.47\textwidth]{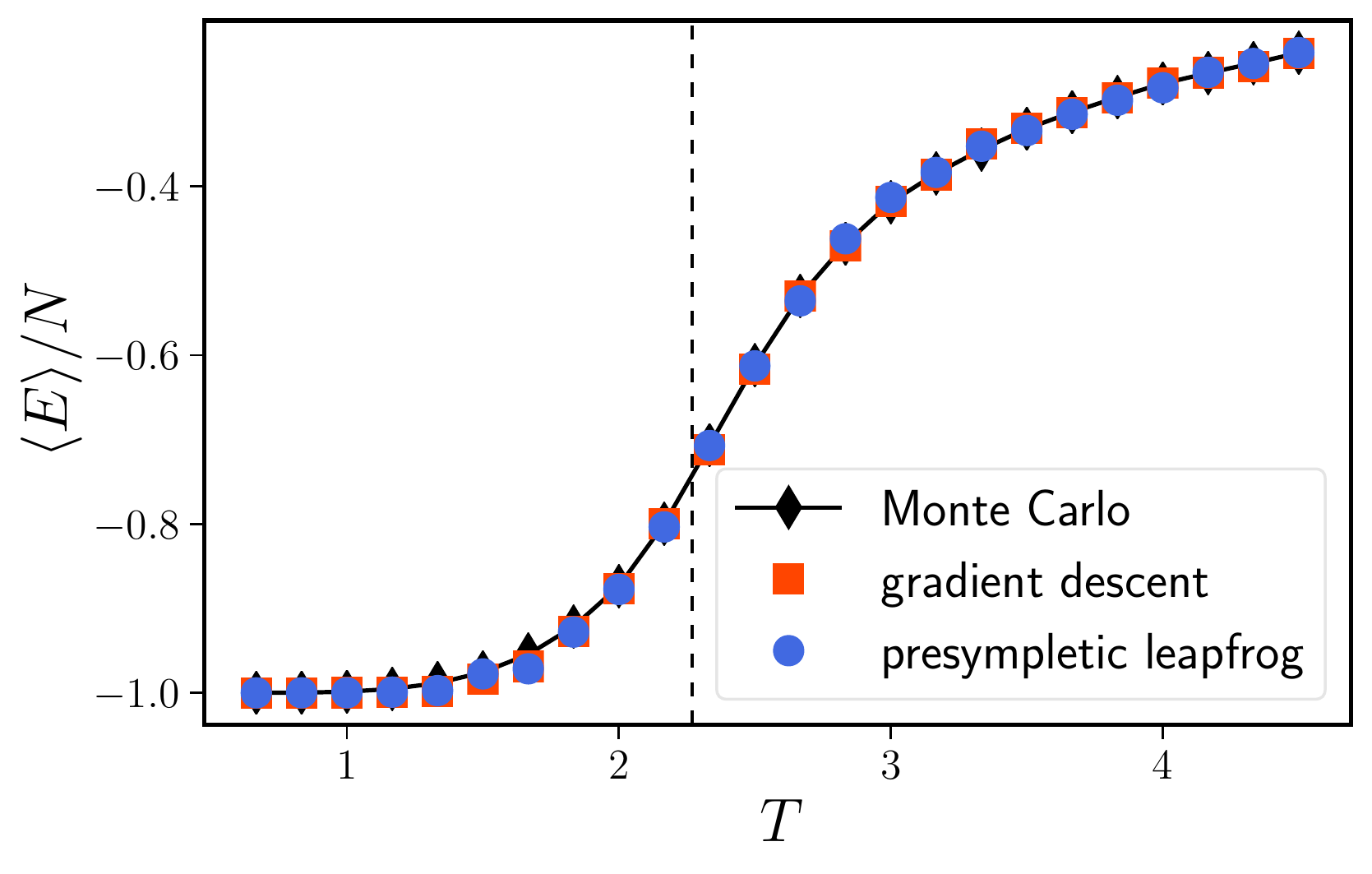}~~~~%
\includegraphics[width=.47\textwidth]{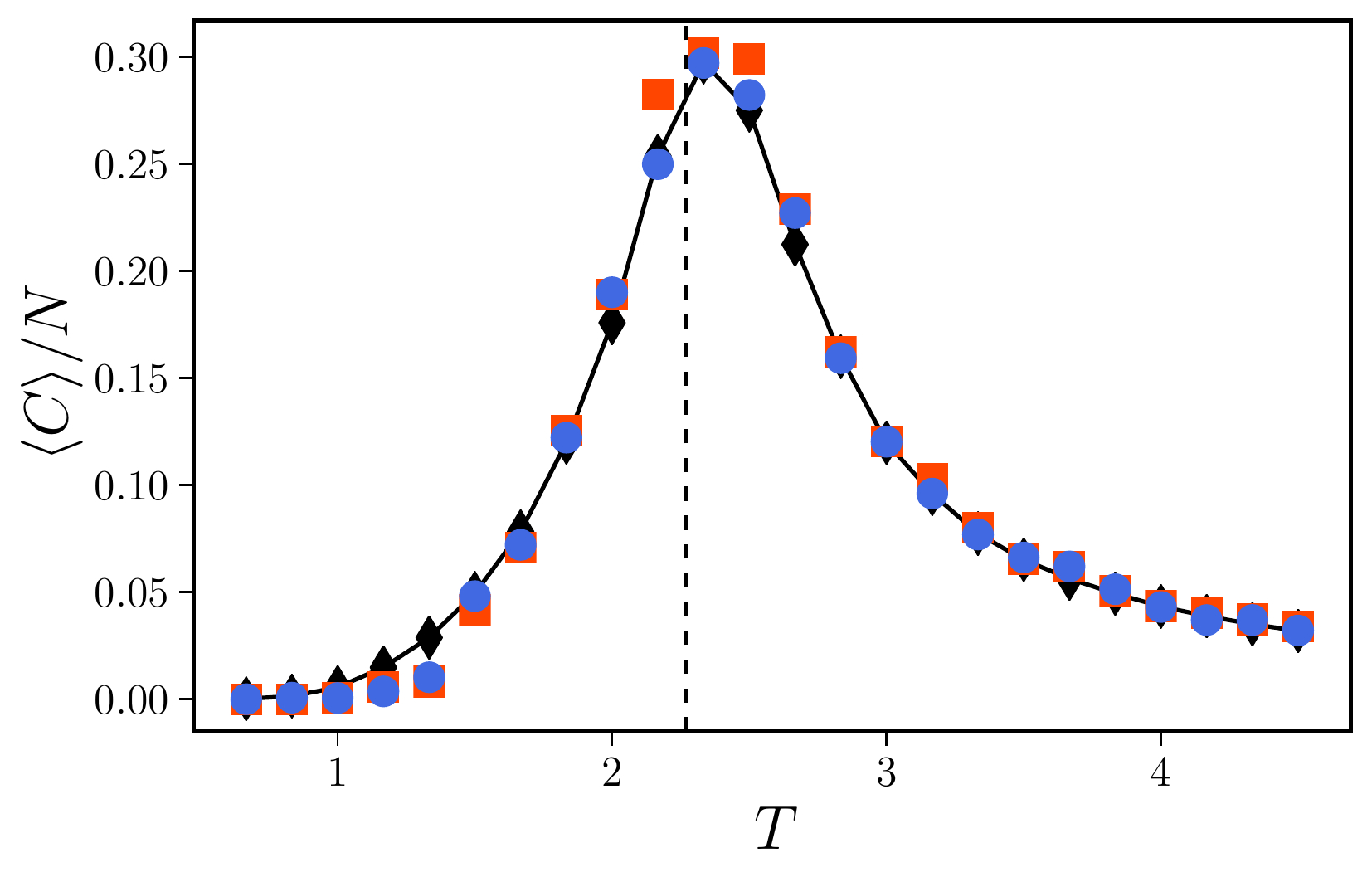}
\caption{Learning a 2D Ising model ($8\times 8$ lattice) with a restricted
Boltzmann machine (RBM) trained with
gradient descent (standard approach)
and the presymplectic leapfrog method \eqref{quad_leap}.
The RBM was trained with only
$2\times 10^3$ data points. We include a comparison with the
Metropolis Monte Carlo estimates to the average energy and heat capacity
\eqref{observables} using $10^5$ configurations.
The vertical dashed line indicates the theoretical value of the critical
temperature $T_c \approx 2.27$.}
\label{ising_plot}
\end{figure}

Following \cite{Torlai:2016,Morningstar:2018}, we consider the Ising model
on a lattice of size
$N = 8 \times 8$ and periodic boundary conditions.
We generate $10^5$ Ising configurations using a standard Monte Carlo
simulation  with the Metropolis algorithm; we use $N^3$ equilibration steps
for several values of  the temperature in the range $T \in [0.5, 4.5]$.
Then, for each temperature $T$, we feed $2 \times 10^3$ random
samples of these states
to the RBM which is optimized with gradient descent
and the presymplectic leapfrog method \eqref{quad_leap} (with $M = I$
for simplicity). For gradient descent we use a step size $h = 5\times 10^{-3}$
while for the leapfrog we let $h = 2.5\times 10^{-2}$ and set the
damping coefficient $\gamma$ according to
$\mu \equiv e^{-\gamma h /2}$ with $\mu = 0.98$ (these values were chosen via a rough grid search).  To compute (stochastic)
gradients we use a ``batch size'' of $100$ and run these algorithms for
$2 \times 10^3$ iterations.
The number of neurons in the hidden layer of the RBM is set to
be $400$ and  we use only one step $k=1$ in the contrastive gradient computation (more
steps yield better results but are computationally demanding).
The RBM estimates for the energy and specific heat
are reported in Fig.~\ref{ising_plot}.\footnote{These results
can be compared with
\cite{Torlai:2016,Morningstar:2018} where more extensive (and expensive)
experiments were performed for this same problem.
}
Note that the presymplectic leapfrog \eqref{quad_leap} yields
improvements over gradient descent.  Indeed, to illustrate its faster
convergence,  in Fig.~\ref{convergence} we show a plot of the loss function
during training for  a single temperature; analogous results
hold for other values. We also tested
Nesterov's method on this problem and the results were essentially the
same as the ones reported for the presymplectic leapfrog.

\begin{figure}
\centering
\includegraphics[width=.5\textwidth]{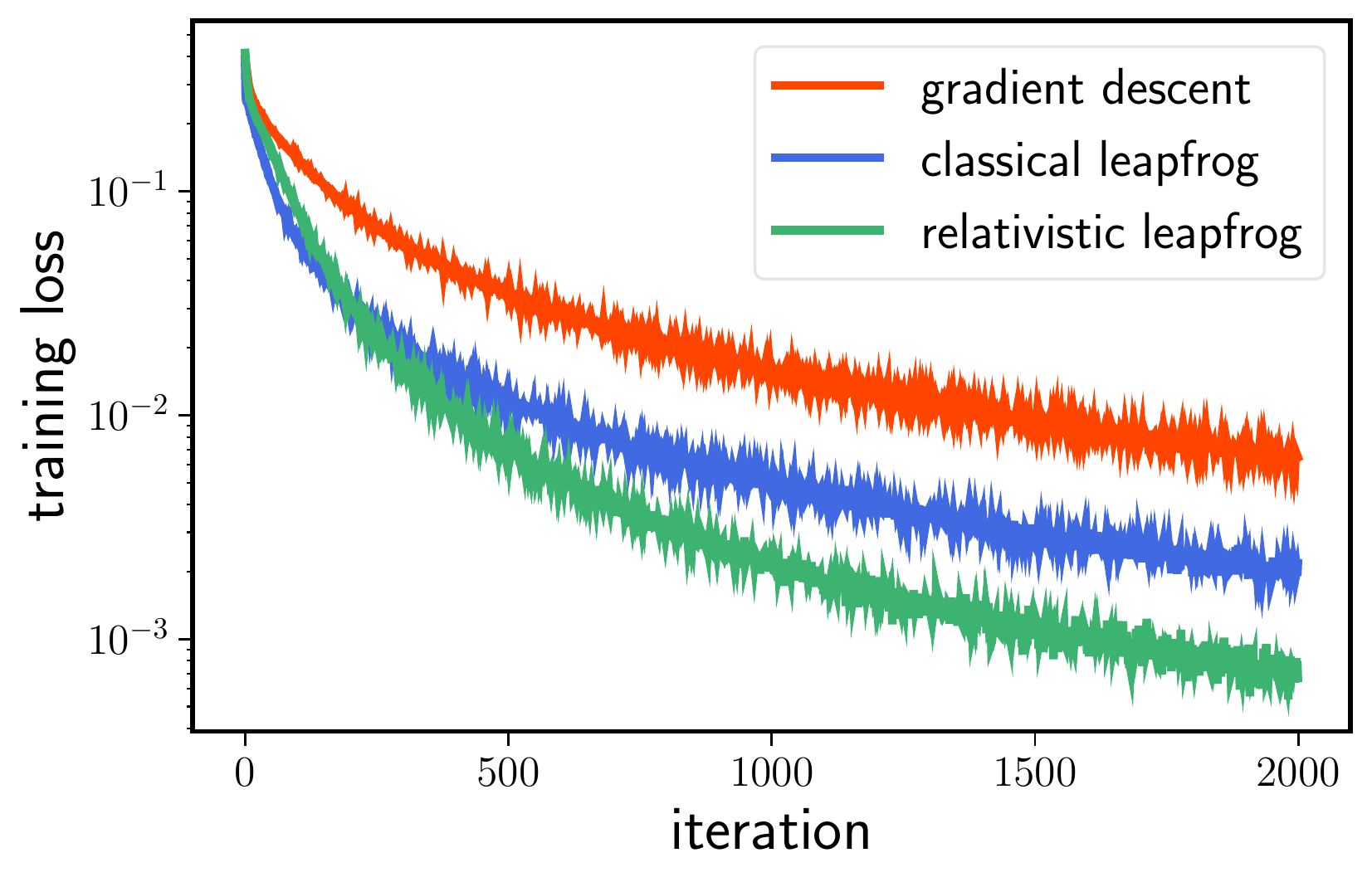}
\caption{Convergence rate in training an RBM with
gradient descent and the presymplectic leapfrog method \eqref{quad_leap}
for one data point ($T\approx 2.83$) of Fig.~\ref{ising_plot}. We also include
results for a method obtained from a dissipative relativistic system
\eqref{rel_ham}.}
\label{convergence}
\end{figure}

Finally, to illustrate that different physical systems can also lead to
feasible optimization methods, we consider the presymplectic
leapfrog \eqref{presymp_leapfrog2} with a relativistic kinetic energy.
That is, we consider the following Hamiltonian
(see \ref{generalized_conformal} and \cite{Franca:2019}):
\be \label{rel_ham}
H = e^{\gamma t}\sqrt{e^{-2 \gamma t} c^2 p^j p_j + m^2 c^4 } +
e^{\gamma t} f(q).
\ee
We set $m = 1$,  $h = 2.8 \times 10^{-2}$, $\mu = e^{-\gamma h/2} = 0.99$
and we fix the
speed of light to $c=5$---in the algorithm $c$ can be considered a free
parameter that controls the kinetic energy. In Fig.~\ref{convergence} we see
that this method has faster convergence compared to the classical system. (We also considered the
entire curves of Fig.~\ref{ising_plot}
but the results were not significantly different from the classical case.)

\section{Discussion} \label{conclusion}

We have introduced \emph{``presymplectic integrators''} as a class of discretizations that are suitable for simulating explicitly time-dependent or nonconservative---and in particular dissipative---Hamiltonian systems.  This framework accommodates a large class of dissipative dynamical systems that are appropriate for applications in optimization and machine learning.  We have also shown that, besides preserving the underlying \emph{presymplectic geometry} of nonconservative Hamiltonian systems, these methods nearly preserve the Hamiltonian and exhibit long-term stability despite the absence of a conservation law; see Theorem~\ref{preserv_ham_decay}. This extends into nonconservative settings the most important property of symplectic integrators which are restricted to conservative systems.
Thus, our approach and theoretical conclusions are applicable to a variety of scientific disciplines where the simulation of nonconservative or dissipative systems play an important role, such as nonequilibrium statistical physics, thermodynamics of open systems, complex and nonlinear systems, economics, and so forth.

Focusing on optimization, we showed that, for a general class of dissipative systems arising from a Hamiltonian in the form \eqref{gen_ham}, presymplectic integrators are able to preserve the continuous-time rates of convergence  up to a negligible error, as long as certain conditions are satisfied; see Theorem~\ref{main_theorem}.
This  provides a systematic and first-principles approach to deriving ``rate-matching'' discretizations,
thereby obviating the need for a discrete-time convergence analysis.\footnote{Of course, if one wants to study fine details of a specific algorithm's performance, a discrete analysis, together with an appropriate backward error analysis, may be necessary.}
As a concrete example, we considered the Bregman Hamiltonian, providing general methods from which specific optimization algorithms can be derived, in both the separable and the nonseparable case.
Our theoretical  conclusions and the feasibility of this approach
were also well-supported by numerical experiments.

We comment on some problems that might deserve further study.
First, we showed that condition \eqref{eta_bound} is essential and in
particular leads to the damping strategies
\eqref{eta1_large} and \eqref{eta1_small} related to the heavy ball and
Nesterov's method, respectively.  We also argued that a more elaborate
choice \eqref{damp_between} may be beneficial. Thus, finding the \emph{optimal
damping} for a given class of problems and relating to the amount of energy
being dissipated seems an interesting  problem.
More generally, finding physical systems that may provide a faster convergence
is interesting---we considered the relativistic kinetic energy besides
the classical one but other choices may be possible.
Second, it would be appealing to find numerical schemes with a better
global error---see equation \eqref{lipschitz_flow}---or obtain improved
bounds for existing methods, since this would automatically relax the
condition \eqref{eta_bound}.
Finally, our geometric construction is fully general and one can also
consider the simulation of dissipative systems on arbitrary (smooth) curved
manifolds. Thus, our framework can also be used to solve optimization
problems over a Riemannian manifold by including the
appropriate metric in the kinetic part of the Hamiltonian.

\bigskip

\subsubsection*{Acknowledgements}
\vspace{-1em}
We wish to thank Jelena Diakonikolas and Michael Muehlebach for helpful discussions.  This work was supported by grant ARO MURI W911NF-17-1-0304.

\bigskip

\appendix

\section{Background on differential geometry} \label{diff_geo_sec}

In this section we provide a brief overview of those elements of differential geometry that we need for our results. For a fuller presentation we refer to any of the excellent textbooks on the subject (e.g. \cite{Arnold,Marsden,Berndt,Aebischer,Sternberg}).

Let $\M \equiv \M^{\n}$ be a smooth $\n$-dimensional manifold. Let $p \in \M$ and $(\R, x)$ be a \emph{chart} so that $p$ can be assigned local coordinates,
$x(p)\equiv (x^1(p),\dotsc,x^\n(p) )$ in $\mathbb{R}^\n$. The coordinates $x$ and the point $p$ are used interchangeably, and we often refer to the former to indicate a point on the manifold. To each $x\in \M$ there is an associated vector space $T_x \M$ called the \emph{tangent space}. The coordinates $x$ induces a basis $\partial_1,\dotsc,\partial_\n$ in $T_x\M$, where $\partial_j \equiv \tfrac{\partial}{\partial x^j}$. Thus, $V\vert_x = V^j(x) \partial_j$ is a representation of
the vector $V\vert_x \in T_x \M$.\footnote{ \label{einstein_notation}
We use Einstein's summation convention throughout, where a pair of upper and lower indices are summed over; e.g.
$X^j \partial_j  \equiv \sum_{j=1}^\n X^j \partial_j$, $\alpha_{jk} T^{\ell jk} \equiv \sum_{j=1}^{\n}\sum_{k=1}^\n \alpha_{jk} T^{\ell jk}$, etc. Upper indices denote components of vectors---also called contravariant vectors---in $T_x\M$, while lower indices denote components of dual vectors---also called covectors---which belong to the cotangent space $T_x^*\M$.
}
The collection of all tangent vectors of $\M$ form the \emph{tangent bundle} $T \M$. One can then represent a (contravariant) vector field by the differential operator
\be \label{vector_field}
  V(x ) \equiv V^j(x) \partial_j,
\ee
which can be seen as a cross section of the tangent bundle $T\M$.
For a given $x$, the vector field $V$ assigns a single vector, $V(x) = V\vert_x$.
Note that a point in $T\M$ has $2\n$ coordinates, $x^1,\dotsc,x^\n, V^1,\dotsc,V^\n$.

To each $V\vert_x $ there is an associated dual vector, $\alpha\vert_x : T_x M \to \mathbb{R}$; i.e. $\alpha$ is a linear functional. Dual vectors are called covectors or 1-forms and they live on the \emph{cotangent space} denoted by $T_x^* \M$, which is isomorphic to $T_x\M$. The collection of all 1-forms at every point of $\M$ forms the \emph{cotangent bundle} $T^* \M$. The coordinate basis $x$ induces a dual basis  $dx^1,\dotsc,dx^\n$ in $T^*_x \M$, defined by $dx^j(\partial_k) = \partial_k(dx^j) = \delta^j_k $, where $\delta$ is the Kronecker delta. Similarly to \eqref{vector_field} one can now represent a 1-form  field $\alpha  \in T^*\M$ as
\be \label{dual_vector_field}
\alpha(x) = \alpha_j(x) dx^j,
\ee
which is a cross section of the cotangent bundle $T^*\M$.
The action of dual vectors is thus $\alpha( V )\vert_x = V(\alpha)\vert_x = \alpha_j(x) V^j(x)$.
Note that $T^* \M$ is a $2\n$-dimensional space where a point has coordinates
$x^1,\dotsc,x^\n,\alpha_1,\dotsc,\alpha_\n$. The cotangent bundle $T^*\M$ is very special since it can be shown that it is itself a symplectic manifold (see Theorem~\ref{cotangent_symplectic}).

A general \emph{tensor} $T$ of rank $(p,q)$ is a multilinear map
$T: \bigotimes^p T_x^* \M \bigotimes^q T_x \M \to \mathbb{R}$. In a coordinate basis it is written as
\be
  T = \tensor{T}{^{j_1} ^\dotsc ^{j_p} _{k_1} _\dotsc _{k_q}} \partial_{j_1} \dotsm \partial_{j_p} dx^{k_1} \dotsm dx^{k_q}.
\ee
In calculations, it is useful to focus on the
components $\tensor{T}{^{j_1} ^\dotsc ^{j_p} _{k_1} _\dotsc _{k_q}}$  alone and omit the basis altogether. A \emph{contravariant tensor} is a tensor of rank $(p, 0)$, for some $p\ge 1$, and  a \emph{covariant tensor} is a tensor of rank $(0,q)$, for some $q \ge 1$.
A \emph{$q$-form}  is a $(0,q)$-tensor which is totally \emph{antisymmetric} in its indices. In a basis it is denoted as
\be
  \alpha  = \dfrac{1}{q!} \alpha_{j_1 \dotsc j_q} dx^{j_1} \wedge \dotsm \wedge dx^{j_q},
\ee
where $\wedge$ denotes the \emph{exterior product}.\footnote{For two 1-forms $\alpha, \beta \in T_x^*\M$ we have
$(\alpha\wedge \beta)(v,w) \equiv \alpha(v)\beta(w) - \beta(v)\alpha(w)$, for $v,w \in T_x\M$.} Given another $p$-form $\beta$, one can compose the $(p+q)$-form $\alpha\wedge\beta$ which obeys
\be
  \alpha \wedge \beta= (-1)^{pq} \, \beta \wedge \alpha.
\ee
It is useful to introduce a notation to denote the space of $q$-forms at $x$, namely $\bigwedge^q T_x^*\M$, and as before we obtain a bundle $\bigwedge^q T^*\M$ of $q$-form fields by allowing the coefficient $\omega_{j_1 \dotsc j_q}(x)$ to depend on $x$.

Another important operation is the \emph{exterior derivative}:
\be
d : \bigwedge^{q} T^*\M \to \bigwedge^{q+1} T^*\M,
\ee
which can be defined componentwise as
\be
  d\alpha(x)  \equiv \dfrac{1}{q!} \dfrac{\partial \alpha_{j_1\dotsc j_q}(x)}{\partial x^k} dx^k \wedge dx^{j_1} \wedge \dotsm \wedge dx^{j_q} .
\ee
This is a linear operation, $d(c_1 \alpha + c_2 \beta) = c_1 d\alpha+ c_2 d \beta$,
for any forms $\alpha$ and $\beta$.
Its main properties are given by the equation
\be\label{wedge_prop1}
d(\alpha \wedge \beta) = (d\alpha) \wedge \beta + (-1)^q \alpha\wedge (d\beta),
\ee
and by the identity
\be \label{poinc_lemma}
d^2 = d \circ d = 0.
\ee
A differential form $\omega$ is said to be \emph{closed} if $d \omega = 0$. A differential $q$-form $\omega$ is said to be \emph{exact} if $\omega = d \lambda$ for some  $(q-1)$-form $\lambda$.  Trivially, from \eqref{poinc_lemma} every exact form is closed.  The \emph{Poincar\' e lemma} ensures the converse, namely that
every closed form is also exact.\footnote{This holds for contractible manifolds, which is the case for smooth manifolds as considered in this paper.}

Given a vector $v \in T_x \M$ and a $q$-form $\alpha \in \bigwedge^q T_x^*\M$, the \emph{interior product} $i_v \alpha$ is a $(q-1)$-form defined by
\be
(i_v\alpha)(v_2,\dotsc,v_q) \equiv \alpha(v, v_2,\dotsc,v_q).
\ee
In components, this is
simply the contraction $(i_v \alpha)_{j_2 \dotsc j_q} = v^{k} \alpha_{k j_2\dotsc j_q}$. The interior product is also linear and satisfies an analogous relation to \eqref{wedge_prop1}:
\be
i_v(\alpha\wedge \beta) = (i_v\alpha)\wedge \beta + (-1)^q \alpha \wedge (i_v \beta).
\ee

Since one can only operate on elements of the same vector space, it is necessary to introduce a mapping that makes it possible to move  geometric objects over the manifold. In particular, given a function
$F : \M \to \N$
between two manifolds $\M$ and $\N$,
for any function $g : \N \to \mathbb{R}$
one defines the \emph{pushforward}---also called the differential---of $v \in T_x \M$ to be the vector $F_* v \in  T_{F(x)}\N$ defined by the operation
\be \label{push_forward}
(F_* v)(g) \equiv v(g\circ F) .
\ee
The \emph{pullback} goes in the opposite direction, i.e., given a $q$-form $\alpha \in T^*_{F(x)} \N$ we obtain the $q$-form $F^* \alpha \in T_x^* \M$ through
\be \label{pullback}
F^* \alpha(v_1,\dotsc,v_q) \equiv \alpha(F_* v_1, \dotsc, F_* v_q),
\ee
for $v_j \in T_x \M$, $j=1,\dotsc,q$. The pullback allows us to move $q$-forms over the manifold. By introducing the concept of a flow, $\flow_t: \M \to \M$ induced by a vector field $X$ of $T\M$, i.e., $\flow_t = e^{t X}$, one can define the  \emph{Lie derivative} of a $q$-form as in \eqref{LieDer}.
A flow $\flow_t$ is  a diffeomorphism and thus $\flow_t^* = (\flow_{-t})_{*}$. In this case, one can consider the pullback of not only  $q$-forms but arbitrary $(p,q)$-tensors. Thus,  the Lie derivative of a general tensor field $T$ can be defined by
\be \label{lie_tensor}
\Lie_{X} T\vert_{x} \equiv \left.\dfrac{d}{dt}\right\vert_{t=0} \flow^*_t T\vert_{\flow_t(x)} .
\ee
Given a differentiable form $\omega$ and a vector field $X$, we say that $X$ preserves $\omega$ if and only if
\be
\Lie_X \omega = 0.
\ee
From the definition \eqref{lie_tensor} this implies
\be
\flow_t^* \omega = \omega.
\ee
In this sense,
any map $\phi : \M \to \N$ is said to be \emph{canonical}  if it preserves
$\omega$, i.e. $\phi^* \omega = \omega$.   In the case of Hamiltonian systems $\omega$ is the symplectic 2-form, and  Hamiltonian flows generate canonical transformations, which are symplectomorphisms.

A very useful formula is Cartan's magic formula:
\be
\Lie_X \alpha = d \circ i_X \alpha  + i_X \circ d \alpha
\ee
for any differentiable form $\alpha$. Some other useful formulas are
\begin{align}
  \Lie_X(\alpha\wedge \beta) &= (\Lie_X \alpha) \wedge \beta + \alpha \wedge (\Lie_X\beta), \\
  \Lie_{[X,Y]}\alpha &= 
  [\Lie_X,\Lie_Y]\alpha, \\
  \Lie_X \circ d &= d\circ \Lie_x, \\
  i_{[X,Y]} &= \Lie_X\circ i_Y - i_Y\circ\Lie_X .
\end{align}
Here, $[\Lie_X,\Lie_Y] \equiv \Lie_X \Lie_Y - \Lie_Y \Lie_X$ is the \emph{Lie bracket}. The same
definition holds for $[X,Y] = XY  - YX$, and in this case one refers to $[\cdot,\cdot]$ as the \emph{commutator} of two vector fields.

\section{Generalized conformal Hamiltonian systems} \label{generalized_conformal}

Conformal Hamiltonian systems \cite{McLachlan:2001} provide an alternative approach to introducing dissipation into Hamiltonian
systems. In this approach, one modifies Hamilton's equations directly by adding a linear term in the momentum.
It is easy to construct structure-preserving discretizations for these systems since one can split the system into conservative and dissipative parts, then apply a standard symplectic integrator to the former, while integrating the latter exactly.
This approach has been recently explored in optimization \cite{Franca:2019}.
The purpose of this section is to show that a---nonautonomous---generalization of conformal Hamiltonian systems correspond to a particular case of the explicit time-dependent Hamiltonian formalism introduced in section~\ref{ham_sys_sec}.  As a consequence, one can construct (generalized) conformal symplectic integrators from presymplectic integrators (see Definition~\ref{presymp_int}).

Consider a time-independent Hamiltonian, $H=H(q,p)$, and assume a modified form of Hamilton's equations given by
\be \label{gen_conf}
\dfrac{dq^j}{dt} = \dfrac{\partial H}{\partial p_j}, \qquad
\dfrac{dp_j}{dt} = - \dfrac{\partial H}{\partial q^j} - \gamma(t) p_j . 
\ee
The conformal case \cite{McLachlan:2001} assumes that the damping coefficient $\gamma(t) = \gamma$ is constant. In this formulation, the Hamiltonian vector field is
\be
X_H = \dfrac{\partial H}{\partial q^i} \dfrac{\partial}{\partial q^i} + \dfrac{\partial H}{\partial p_i} \dfrac{\partial}{\partial p_i},
\ee
and by the same geometric approach previously discussed one can see that the equations of motion \eqref{gen_conf} are equivalent to
\be
i_{X_H}(\omega) = - dH - \gamma(t) \lambda,
\ee
with $\omega$ defined in \eqref{symp} and where $\lambda \equiv p_j dq^j$ is the Liouville-Poincar\' e 1-form.  From Cartan's formula \eqref{cartan} we conclude that the symplectic structure contracts as
\be \label{omega_conf}
\Lie_{X_H} \omega = - \gamma(t) \omega,
\ee
or equivalently\footnote{From this one concludes that the phase-space volumes contract as $\Lie_{X_H}\vol^{2\n} = -\n \gamma(t) \vol^{2\n}$, so that
$\flow_t^* \vol^{2n} = e^{- n C(t)} \vol^{2n}$. This is a dissipative version of Liouville's theorem;
note the dimension dependency.}
\be
\flow_t^* \omega = e^{-\eta(t)} \omega, \qquad \eta(t) \equiv \int^t \gamma(t') dt'.
\ee
Finally, in this setting $H$ is the energy of the system and it dissipates as
\be \label{energy_conf}
  \dfrac{dH}{dt} = - \gamma(t) \dfrac{\partial H}{ \partial p_j } p_j.
\ee

We now show that generalized conformal Hamiltonian systems corresponds to a particular case of the time-dependent Hamiltonian formalism.
Define the time-dependent Hamiltonian $K$ as
\be \label{new_ham_time}
  K(t, Q,P) \equiv e^{\eta(t)} H\big(Q, e^{-\eta(t)} P\big), \qquad Q \equiv q, \qquad P \equiv e^{\eta(t)} p,
\ee
where $H$ is the original Hamiltonian of \eqref{gen_conf}. Now the
standard Hamilton's equations  yield
\be
  \dfrac{dQ^j}{dt} = \dfrac{\partial K}{\partial P_j} = \dfrac{\partial H(q,p)}{\partial p_j}, \qquad
  \dfrac{dP_j}{dt} = -\dfrac{\partial K}{\partial Q^j} = - e^{\eta(t)} \dfrac{\partial H(q,p)}{\partial q^j},
\ee
which when written in terms of $(q,p)$ are precisely \eqref{gen_conf}.
Therefore, one can  construct discretizations that preserve the contraction of the symplectic form \eqref{omega_conf}
through presymplectic integrators (see Definition~\ref{presymp_int}) with the explicit time-dependent Hamiltonian $K$.

\section{Constructing presymplectic integrators} \label{constructing_symp}

As stated in Definition~\ref{presymp_int}, a presymplectic integrator is a reduction of a higher-dimensional symplectic integrator under the gauge fixing \eqref{q0} and \eqref{p0}. To follow this prescription one must perform three simple steps:
\begin{enumerate}
\item Choose a \emph{symplectic integrator}---for a conservative Hamiltonian system---and apply it to the Hamiltonian system \eqref{ham_big}. This results into updates for $(q^0, p_0)$ and for the spatial components $(q^j, p_j)$;
\item Set $q^0 = t$ and ignore $p_0$ completely---$p_0$ is just the actual value of the Hamiltonian as a function of time and does not participate in the dynamics, neither in the numerical procedure;
\item Set $s = t$.
\end{enumerate}
While these formal steps make clear that we are respecting the symplectification prescription previously discussed, in practice these three steps can be
reduced to the following:
\begin{itemize}
\item Apply any ``symplectic integrator'' to the time-dependent
Hamiltonian $H(t,q,p)$ in the ``natural way.'' By this we mean to simply
include additional updates for the time variable $t$ with the same
rule as the coordinates $q$.
\end{itemize}
We will provide some explicit examples below that should make this
 clear.

\subsection{Presymplectic Euler}
For a conservative Hamiltonian system one has the following version of the
symplectic Euler method which is order $r=1$ \cite{Hairer} ($\ell = 0,1,\dotsc$ is the iteration number and $h$ the step size):
\be
\label{symp_euler}
\begin{split}
p_{\m+1} &= p_\m - h \nabla_q H(q_{\m}, p_{\m+1}), \\
q_{\m+1} &= q_{\m} + h\nabla_p H(q_{\m}, p_{\m+1}).
\end{split}
\ee
Considering a time-dependent Hamiltonian $H(t, q, p)$,
since $t$ must be treated in the same way as the coordinates $q$
according to the above
discussion, we immediately obtain the
\emph{presymplectic Euler} method given by
\be
\label{presymp_euler}
\begin{split}
p_{\m+1} &= p_\m - h \nabla_q H(t_\m, q_\m, p_{\m+1}), \\
t_{\m+1} &= t_\m + h, \\
q_{\m+1} &= q_\m + h\nabla_p H(t_\m, q_\m, p_{\m+1}). \\
\end{split}
\ee
Note how we have simply added an update for $t$ following the same structure as the update for $q$.
There is also the adjoint
of \eqref{symp_euler} given by \cite{Hairer}
\be
\label{symp_euler2}
\begin{split}
q_{\m+1} &= q_\m + h\nabla_p H(q_{\m+1}, p_\m), \\
p_{\m+1} &= p_\m - h \nabla_q H(q_{\m+1}, p_{\m}).
\end{split}
\ee
From this we obtain an alternative to \eqref{presymp_euler}
which is
\be \label{presymp_euler2}
\begin{split}
t_{\m+1} &= t_\m + h , \\
q_{\m+1} &= q_\m + h\nabla_p H(t_{\m+1}, q_{\m+1}, p_{\m}) , \\
p_{\m+1} &= p_\m - h \nabla_q H(t_{\m+1}, q_{\m+1}, p_{\m}) . \\
\end{split}
\ee
Both of these methods, namely \eqref{presymp_euler} and \eqref{presymp_euler2}, are of order $r=1$.
Note that for an arbitrary Hamiltonian $H$ in general they are implicit, i.e. nonlinear
equations have to be solved to obtain either $q_{\m+1}$ or $p_{\m+1}$.
However, when the Hamiltonian is separable in the form
\be \label{separable}
H = T(t,p) + V(t,q)
\ee
these methods become completely explicit in all variables resulting in cheap implementations.

\subsection{Presymplectic leapfrog}
One of the versions of the
leapfrog method for a conservative Hamiltonian is given by \cite{Hairer}
\begin{equation}
\label{leapfrog}
\begin{split}
p_{\m + 1/2} &= p_\m - (h/2) \nabla_q H(q_\m, p_{\m+1/2}), \\
q_{\m+1} &= q_\m + (h/2)\left(\nabla_p H(q_\m, p_{\m+1/2}) + \nabla_p H(q_{\m+1}, p_{\m+1/2})\right), \\
p_{\m + 1} &= p_{\m+1/2} - (h/2) \nabla_q H(q_{\m+1}, p_{\m+1/2}).
\end{split}
\end{equation}
This method is obtained by composing \eqref{symp_euler} and
\eqref{symp_euler2}, each with step size $h/2$ and in this order.
This method is symmetric and thus of order $r=2$.
For a time-dependent Hamiltonian $H(t,q,p)$ we include appropriate
updates for time and obtain the following \emph{presymplectic leapfrog}
method:
\be
\label{presymp_leapfrog}
\begin{split}
p_{\m + 1/2} &= p_\m - (h/2) \nabla_q H(t_\m, q_\m, p_{\m+1/2}), \\
t_{\m+1} &= t_\m + h, \\
q_{\m+1} &= q_\m + (h/2)\left(\nabla_p H(t_\m, q_\m, p_{\m+1/2}) + \nabla_p H(t_{\m+1}, q_{\m+1}, p_{\m+1/2})\right), \\
p_{\m + 1} &= p_{\m+1/2} - (h/2) \nabla_q H(t_{\m+1}, q_{\m+1}, p_{\m+1/2}).
\end{split}
\ee
There is also the adjoint version of the leapfrog which
is obtained by composing \eqref{symp_euler2} followed
by \eqref{symp_euler} instead, i.e. \cite{Hairer}
\be
\label{leapfrog2}
\begin{split}
q_{\m+1/2} &= q_\m + (h/2) \nabla_p H(q_{\m+1/2}, p_\m), \\
p_{\m+1} &= p_\m - (h/2) \big( \nabla_q H(q_{\m+1/2}, p_\m) + \nabla_q H(q_{\m+1/2}, p_{\m+1}) \big), \\
q_{\m+1} &= q_{\m+1/2} + (h/2) \nabla_p H(q_{\m+1/2}, p_{\m+1}).
\end{split}
\ee
This leads to an alternative version of the presymplectic leapfrog given by
\begin{equation}
\label{presymp_leapfrog2}
\begin{split}
t_{\m+1/2} &= t_\m + h/2, \\
q_{\m + 1/2} &= q_\m + (h/2) \nabla_p H(t_{\m+1/2}, q_{\m+1/2}, p_{\m}), \\
p_{\m+1} &= p_\m - (h/2)\left(\nabla_q H(t_{\m+1/2}, q_{\m+1/2}, p_{\m}) + \nabla_q H(t_{\m+1/2}, q_{\m+1/2}, p_{\m+1})\right), \\
t_{\m+1} &= t_{\m+1/2} + h/2, \\
q_{\m + 1} &= q_{\m+1/2} + (h/2) \nabla_p H(t_{\m+1/2}, q_{\m+1/2}, p_{\m+1}).
\end{split}
\end{equation}
Again, in general these methods are implicit but for
a separable Hamiltonian \eqref{separable} they become completely
explicit in all variables.
Moreover, only one gradient computation per iteration is necessary
even though these methods are of order $r=2$.

Let us comment on  a useful numerical trick.
In practice, for cases such as \eqref{quad_leap}, large exponentials
can be numerically unstable.  This can be avoided by redefining
variables in the algorithm. For instance, in the case of \eqref{quad_leap} one can introduce $\tilde{p}_{\m}  \equiv  e^{-\eta_2(t_\m)} p_\m$ to rewrite the updates as
\be
\begin{split}
\tilde{p}_{\m +1/2} &= e^{-\Delta_{h/2}^{(2,2)}(t_\m) } \left(
\tilde{p}_\m - (h/2)\nabla f(q_\m) \right), \\
t_{\m + 1} &= t_\m + h, \\
q_{\m+1} &= q_\m + (h/2)\left( e^{\Delta_{h/2}^{(2,1)}(t_\m)} +
e^{-\Delta_{h/2}^{(1,2)}(t_{\m+1/2})} \right) \tilde{p}_{\m+1/2}, \\
\tilde{p}_{\m+1} &= e^{-\Delta_{h/2}^{(2,2)}(t_{\m+1/2})} \tilde{p}_{\m+1/2} - (h/2) \nabla f(q_{\m+1}),
\end{split}
\ee
where $\Delta^{(a,b)}_{h}(t) \equiv \eta_a(t + h) - \eta_b(t )$.
Note that only finite differences appear in the exponentials
which prevent very large (or small) numbers in a numerical implementation.
For instance, in the case where $\eta_2 = \eta_1 = \gamma t$, which
was used in the experiments of Section~\ref{rbm_sec}, we have the following
method based on the presymplectic leapfrog:
\be
\begin{split}
\tilde{p}_{\m +1/2} &= \mu \left(
\tilde{p}_\m - (h/2)\nabla f(q_\m) \right), \\
q_{\m+1} &= q_\m + h \cosh ( -\log \mu ) \tilde{p}_{\m+1/2}, \\
\tilde{p}_{\m+1} &= \mu \tilde{p}_{\m+1/2} - (h/2) \nabla f(q_{\m+1}),
\end{split}
\ee
where $\mu \equiv e^{-\gamma h/2}$ (we did not write the
update for $t$ which is not important if one is only interested
in $q$, and in this case one can be even more economical by replacing the last update
into the first). One thus has tuning
parameters $h$ and $\mu$ for the above method.

\subsection{Higher-order methods}
In \cite{Suzuki:1990,Yoshida:1990} an elegant and general approach to
construct arbitrarily higher-order symplectic integrators was presented.
It assumes that a base method
$\flowN_h$ of order $2r$ ($r \ge 1$) is given and
an integrator of order $2r + 2$ is then obtained by the composition
\be \label{yoshida_comp}
\flowN_{\tau_0 h} \circ \flowN_{\tau_1 h} \circ \flowN_{\tau_0 h}
\ee
where
\be
\tau_0 \equiv \dfrac{1}{2  - \kappa}, \qquad
\tau_1 \equiv - \dfrac{\kappa}{2-\kappa}, \qquad
\kappa^{2r+1} \equiv 2.
\ee
One can start with any base method of choice, such as the leapfrog.
From this new integrator of order $2r + 2$ one may proceeds recursively
to construct even higher-order methods.

It is straightforward to adapt this procedure to presymplectic integrators
by carefully adding an update for time $t$.
However, let us mention that the number of gradient computations of the
Suzuki-Yoshida approach \eqref{yoshida_comp} grows very fast with
increasing order, quickly becoming unfeasible.
Moreover, the truncation error of such methods tend to be rather large,
although the fourth-order method obtained from the leapfrog is
competitive and interesting.
(See \cite{McLachlan:2006} and especially \cite{McLachlan:2002b} for an
interesting discussion.)

\bibliography{biblio.bib}

\providecommand{\href}[2]{#2}\begingroup\raggedright\begin{thebibliography}{10}

\bibitem{Candes:2016}
W.~Su, S.~Boyd, and E.~J. Cand{{\`e}}s, ``A differential equation for modeling
  {N}esterov's accelerated gradient method: theory and insights,'' {\em J.
  Mach. Learn. Res.} {\bfseries 17} no.~153, (2016) 1--43.

\bibitem{Wibisono:2016}
A.~Wibisono, A.~C. Wilson, and M.~I. Jordan, ``A variational perspective on
  accelerated methods in optimization,''
  \href{http://dx.doi.org/10.1073/pnas.1614734113}{{\em {Proc. Nat. Acad.
  Sci.}} {\bfseries 113} no.~47, (2016) E7351--E7358}.

\bibitem{Franca:2018b}
G.~Fran{\c c}a, D.~P. Robinson, and R.~Vidal, ``A nonsmooth dynamical systems
  perspective on accelerated extensions of {ADMM},''
  \href{http://arxiv.org/abs/1808.04048}{{\ttfamily arXiv:1808.04048
  [math.OC]}}.

\bibitem{Franca:2018}
G.~Fran{\c c}a, D.~P. Robinson, and R.~Vidal, ``{ADMM} and accelerated {ADMM}
  as continuous dynamical systems,'' in {\em International Conference on
  Machine Learning}, vol.~80, pp.~1559--1567.
\newblock PMLR, 2018.

\bibitem{Franca:2019}
G.~Fran{\c{c}}a, J.~Sulam, D.~P. Robinson, and R.~Vidal, ``Conformal symplectic
  and relativistic optimization,''
  \href{http://dx.doi.org/10.1088/1742-5468/abcaee}{{\em J. Stat. Mech.}
  {\bfseries 2020} no.~12, (2020) 124008}.

\bibitem{Krichene:2015}
W.~Krichene, A.~Bayen, and P.~L. Bartlett, ``Accelerated mirror descent in
  continuous and discrete time,'' in {\em Advances in Neural Information
  Processing Systems}, vol.~28.
\newblock 2015.

\bibitem{Wilson:2016}
A.~C. Wilson, B.~Recht, and M.~I. Jordan, ``A {L}yapunov analysis of momentum
  methods in optimization,'' \href{http://arxiv.org/abs/1611.02635}{{\ttfamily
  arXiv:1611.02635 [math.OC]}}.

\bibitem{Scieur:2017}
D.~Scieur, V.~Roulet, F.~Bach, and A.~d{'}Aspremont, ``Integration methods and
  optimization algorithms,'' in {\em Advances in Neural Information Processing
  Systems}, vol.~30.
\newblock 2017.

\bibitem{Betancourt:2018}
M.~Betancourt, M.~I. Jordan, and A.~C. Wilson, ``On symplectic optimization,''
  \href{http://arxiv.org/abs/1802.03653}{{\ttfamily arXiv:1802.03653
  [stat.CO]}}.

\bibitem{Zhang:2018}
J.~Zhang, A.~Mokhtari, S.~Sra, and A.~Jadbabaie, ``Direct runge-kutta
  discretization achieves acceleration,'' in {\em Advances in Neural
  Information Processing Systems}, vol.~31.
\newblock 2018.

\bibitem{Shi:2019}
B.~Shi, S.~S. Du, W.~J. Su, and M.~I. Jordan, ``Acceleration via symplectic
  discretization of high-resolution differential equations,'' in {\em Advances
  in Neural Information Processing Systems}, vol.~32.
\newblock 2019.

\bibitem{Muehlebach:2019}
M.~Muehlebach and M.~I. Jordan, ``A dynamical systems perspective on {N}esterov
  acceleration,'' in {\em International Conference on Machine Learning},
  vol.~97, pp.~4656--4662.
\newblock PMLR, 2019.

\bibitem{Diakonikolas:2019}
J.~Diakonikolas and M.~I. Jordan, ``Generalized momentum-based methods: A
  {H}amiltonian perspective,''
  \href{http://arxiv.org/abs/1906.00436}{{\ttfamily arXiv:1906.00436
  [math.OC]}}.

\bibitem{Nesterov:1983}
Y.~Nesterov, ``{A method of solving a convex programming problem with
  convergence rate $O(1/k^2)$},'' {\em Soviet Math. Doklady} {\bfseries 27}
  no.~2, (1983) 372--376.

\bibitem{Polyak:1964}
B.~T. Polyak, ``{Some methods of speeding up the convergence of iteration
  methods},'' {\em {USSR Comput. Math. and Math. Phys.}} {\bfseries 4} no.~5,
  (1964) 1--17.

\bibitem{Berndt}
R.~Berndt, {\em {An Introduction to Symplectic Geometry}}.
\newblock American Mathematical Society, 2001.

\bibitem{Aebischer}
B.~Aebischer, M.~Borer, M.~Kalin, and C.~Leuenberger, {\em {Symplectic
  Geometry}}.
\newblock Springer Basel AG, 1994.

\bibitem{Benettin:1994}
G.~Benettin and A.~Giorgilli, ``On the {H}amiltonian interpolation of
  near-to-the-identity symplectic mappings with application to symplectic
  integration algorithms,'' \href{http://dx.doi.org/10.1007/BF02188219}{{\em
  {J. Stat. Phys.}} {\bfseries 74} (1994) 1117--1143}.

\bibitem{Reich:1999}
S.~Reich, ``{Backward error analysis for numerical integrators},''
  \href{http://dx.doi.org/10.1137/S0036142997329797}{{\em {SIAM J. Numer.
  Anal.}} {\bfseries 36} no.~5, (1999) 1549--1570}.

\bibitem{Hairer:1994}
E.~Hairer, ``Backward analysis of numerical integrators and symplectic
  methods,'' {\em {Ann. Numer. Math.}} {\bfseries 1} (1994) 107--132.

\bibitem{Reich}
B.~Leimkuhler and S.~Reich, {\em Simulating Hamiltonian Dynamics}.
\newblock Cambridge University Press, 2004.

\bibitem{Hairer}
E.~Hairer, C.~Lubich, and G.~Wanner, {\em Geometric Numerical Integration}.
\newblock Springer, 2006.

\bibitem{SanzSerna:1992}
J.~M. Sanz-Serna, ``Symplectic integrators for {H}amiltonian problems: An
  overview,'' \href{http://dx.doi.org/10.1017/S0962492900002282}{{\em {Acta
  Numerica}} {\bfseries 1} (1992) 243--286}.

\bibitem{McLachlan:2006}
R.~I. McLachlan and G.~R.~W. Quispel, ``Geometric integrators for {ODE}s,''
  \href{http://dx.doi.org/10.1088/0305-4470/39/19/s01}{{\em {J. Phys. A: Math.
  Gen.}} {\bfseries 39} (2006) 5251--5285}.

\bibitem{Quispel:2018}
A.~Iserles and G.~R.~W. Quispel, ``Why geometric numerical integration?,'' in
  {\em Discrete Mechanics, Geometric Integration and Lie--Butcher Series},
  pp.~1--28.
\newblock Springer International Publishing, 2018.

\bibitem{Moore:2016}
A.~Bhatt, D.~Floyd, and B.~E. Moore, ``Second order conformal symplectic
  schemes for damped {H}amiltonian systems,''
  \href{http://dx.doi.org/10.1007/s10915-015-0062-z}{{\em {J. Sci. Comput.}}
  {\bfseries 66} (2016) 1234--1259}.

\bibitem{Moore:2017}
B.~E. Moore, ``{Multi-conformal-symplectic {PDE}s and discretizations},''
  \href{http://dx.doi.org/10.1016/j.cam.2017.04.008}{{\em {J. Comput. and
  Applied Math.}} {\bfseries 323} (2017) 1--15}.

\bibitem{Moore:2019}
A.~Bhatt and B.~E. Moore, ``{Exponential integrators preserving local
  conservation laws of {PDE}s with time-dependent damping/driving forces},''
  \href{http://dx.doi.org/10.1016/j.cam.2018.12.003}{{\em {J. Comput. and
  Applied Math.}} {\bfseries 352} (2019) 341--351}.

\bibitem{Shang:2020}
X.~Shang and H.~C. {\"{O}}ttinger, ``{Structure-preserving integrators for
  dissipative systems based on reversible-irreversible splitting},''
  \href{http://dx.doi.org/10.1098/rspa.2019.0446}{{\em {Proc. R. Soc. A}}
  {\bfseries 476} (2020) 20190446}.

\bibitem{Nicolis:2019}
T.~Nussle, P.~Thibaudeau, and S.~Nicolis, ``Probing magneto-elastic phenomena
  through an effective spin-bath coupling model,''
  \href{http://dx.doi.org/10.1140/epjb/e2019-90539-6}{{\em Eur. Phys. J. B}
  {\bfseries 92} no.~29, (2019) 1--10}.

\bibitem{Nicolis:2019b}
T.~Nussle, P.~Thibaudeau, and S.~Nicolis, ``Dynamic magnetostriction for
  antiferromagnets,'' \href{http://dx.doi.org/10.1103/PhysRevB.100.214428}{{\em
  Phys. Rev. B} {\bfseries 100} (2019) 214428}.

\bibitem{McLachlan:2001}
R.~McLachlan and M.~Perlmutter, ``Conformal {H}amiltonian systems,''
  \href{http://dx.doi.org/10.1016/S0393-0440(01)00020-1}{{\em J. Geom. and
  Phys.} {\bfseries 39} (2001) 276--300}.

\bibitem{Caldirola:1941}
P.~Caldirola, ``Forze non conservative nella meccanica quantistica,''
  \href{http://dx.doi.org/10.1007/BF02960144}{{\em Nuovo Cim.} {\bfseries 18}
  (1941) 393--400}.

\bibitem{Kanai:1948}
E.~Kanai, ``{On the quantization of the dissipative systems},''
  \href{http://dx.doi.org/10.1143/ptp/3.4.440}{{\em Prog. Theor. Phys.}
  {\bfseries 3} (1948) 440--442}.

\bibitem{Caldeira:1981}
A.~O. Caldeira and A.~J. Leggett, ``{Influence of Dissipation on Quantum
  Tunneling in Macroscopic Systems},''
  \href{http://dx.doi.org/10.1103/PhysRevLett.46.211}{{\em Phys. Rev. Lett.}
  {\bfseries 46} no.~4, (1981) 211--214}.

\bibitem{Hairer:1997}
E.~Hairer and C.~Lubich, ``{The life-span of backward error analysis for
  numerical integrators},'' \href{http://dx.doi.org/10.1007/s002110050271}{{\em
  {Numer. Math.}} {\bfseries 76} (1997) 441--462}.

\bibitem{Sternberg}
S.~Sternberg, {\em {Lectures on Differential Geometry}}.
\newblock Prentice Hall, 1964.

\bibitem{Tao:2016}
M.~Tao, ``Explicit symplectic approximation of nonseparable {H}amiltonians:
  algorithm and long time performance,''
  \href{http://dx.doi.org/10.1103/PhysRevE.94.043303}{{\em Phys. Rev. E}
  {\bfseries 94} (2016) 043303}.

\bibitem{Torlai:2016}
G.~Torlai and R.~G. Melko, ``Learning thermodynamics with {B}oltzmann
  machines,'' \href{http://dx.doi.org/10.1103/PhysRevB.94.165134}{{\em Phys.
  Rev. B} {\bfseries 94} (2016) 165134}.

\bibitem{Carleo:2017}
G.~Carleo and M.~Troyer, ``Solving the quantum many-body problem with
  artificial neural networks,''
  \href{http://dx.doi.org/10.1126/science.aag2302}{{\em Science} {\bfseries
  355} (2017) 602--606}.

\bibitem{Melko:2018}
R.~G. Melko, G.~Carleo, and J.~Carrasquilla, ``Restricted {B}oltzmann machines
  in quantum physics,'' \href{http://dx.doi.org/10.1038/s41567-019-0545-1}{{\em
  Nature Physics} {\bfseries 15} (2019) 887--892}.

\bibitem{Morningstar:2018}
A.~Morningstar and R.~G. Melko, ``Deep learning the {I}sing model near
  criticality,'' {\em J. Mach. Learn. Res.} {\bfseries 18} (2018) 1--17.

\bibitem{Hinton:2002}
G.~E. Hinton, ``Training products of experts by minimizing contrastive
  divergence,'' \href{http://dx.doi.org/10.1162/089976602760128018}{{\em Neur.
  Comput.} {\bfseries 14} no.~8, (2002) 1771--1800}.

\bibitem{Arnold}
V.~I. Arnold, V.~V. Kozlov, and A.~I. Neishtadt, {\em {Mathematical Aspects of
  Classical and Celestial Mechanics}}.
\newblock Springer-Verlag Berlin Heidelberg, 2006.

\bibitem{Marsden}
R.~Abraham and J.~E. Marsden, {\em {Foundations of Mechanics}}.
\newblock Addison-Wesley Publishing Company, Inc., 1978.

\bibitem{Suzuki:1990}
M.~Suzuki, ``Fractal decomposition of exponential operators with applications
  to many-body theories and {M}onte {C}arlo simulations,''
  \href{http://dx.doi.org/10.1016/0375-9601(90)90962-N}{{\em {Phys. Lett. A}}
  {\bfseries 146} (1990) 319--323}.

\bibitem{Yoshida:1990}
H.~Yoshida, ``{Construction of higher order symplectic integrators},''
  \href{http://dx.doi.org/10.1016/0375-9601(90)90092-3}{{\em {Phys. Lett. A}}
  {\bfseries 150} no.~5, (1990) 262--268}.

\bibitem{McLachlan:2002b}
R.~I. McLachlan, ``{Families of high-order composition methods},''
  \href{http://dx.doi.org/10.1023/A:1021195019574}{{\em {Numerical Algorithms}}
  {\bfseries 31} (2002) 233--246}.

\end{thebibliography}\endgroup

\end{document}